\documentclass[11pt]{article}
\usepackage{epsfig,mst-stylefile,amssymb,amsmath,harvard}
\usepackage[svgnames]{xcolor}
\usepackage{xr}

\newcommand{\bbZ}{{\Bbb Z}}
\newcommand{\bbR}{{\Bbb R}}
\newcommand{\bbN}{{\Bbb N}}
\newcommand{\bbC}{{\Bbb C}}

\newcommand{\MSD}{\widehat{\textnormal{MSD}}}

\DeclareMathOperator{\tr}{{\mathrm{tr}}}

\DeclareMathOperator{\diag}{{\mathrm{diag}}}

\newcommand{\abs}[1]{\left|#1\right|}
\newcommand{\norm}[1]{ \| #1 \|}

\renewcommand{\cite}{\citeyear}

\begin{document}

\title{The asymptotic distribution of the pathwise mean squared displacement in single particle tracking experiments
\thanks{The first author was supported in part by the prime awards no.\ W911NF--12--1--0512, Short-Term Innovative Research program, and no.\ W911NF-14-1-0475, both from the Biomathematics subdivision of the Army Research Office. The authors would like to thank John Fricks and two anonymous referees for their comments on this work.}
 }

\author{Gustavo Didier \\ Tulane University \and  Kui Zhang \\ Tulane University}

\bibliographystyle{agsm}

\maketitle

\begin{abstract}
Recent advances in light microscopy have spawned new research frontiers in microbiology by working around the diffraction barrier and allowing for the observation of nanometric biological structures. Microrheology is the study of the properties of complex fluids, such as those found in biology, through the dynamics of small embedded particles, typically latex beads. Statistics based on the recorded sample paths are then used by biophysicists to infer rheological properties of the fluid. In the biophysical literature, the main statistic for characterizing diffusivity is the so-named mean squared displacement ($\MSD$) of the tracer particles. Notwithstanding the central role played by the $\MSD$, its asymptotic distribution in different cases has not yet been established. In this paper, we tackle this problem. We take a pathwise approach and assume that the particle movement undergoes a Gaussian, stationary-increment stochastic process. We show that as the sample and the increment lag sizes go to infinity, the $\MSD$ displays Gaussian or non-Gaussian limiting distributions, as well as distinct convergence rates, depending on the diffusion exponent parameter. We illustrate our results analytically and computationally based on fractional Brownian motion and the (integrated) fractional Ornstein-Uhlenbeck process.  
\end{abstract}


\section{Introduction}

Abbe's diffraction limit stood for more than a hundred years as a barrier for light microscopy. The resolution limit of roughly 250nm (1nm = $10^{-9}$m) is large compared to organelles in biological cells and most nanostructures. However, in the last twenty years advances in light microscopy technology have spawned new research frontiers by allowing for the observation of nanobiological phenomena \textit{in vitro} and \textit{in vivo} up to resolutions of 10--20nm (e.g., Hell \cite{hell:2003,hell:2008}, Betzing et al \cite{betzig:etal:2006}, Rust et al \cite{rust:bates:zhuang:2006}, Hess et al \cite{hess:girirajan:mason:2006}, Westphal et al \cite{westphal:etal:2008}, Berning et al \cite{berning:willig:steffens:dibaj:hell:2012}, Jones et al \cite{jones:shim:he:zhuang:2011}, Huang et al \cite{huang:etal:2013}). Microrheology is a rapidly expanding subfield of nanobiophysics. It consists of the study of the properties of complex fluids, such as those found in biology, through the dynamics of small embedded particles, typically latex beads, tracked and recorded by means of new light microscopy technology. Microrheology is currently the dominant technique in the study of the physical properties of complex biofluids, of the rheological properties of membranes or the cytoplasm of cells, or of the entire cell (Mason and Weitz \cite{mason:weitz:1995}, Wirtz \cite{wirtz:2009}; see Didier et al.\ \cite{didier:mckinley:hill:fricks:2012} for a broad description of the statistical challenges involved).

The characterization of the diffusive behavior of nanometric particles embedded in viscous, Newtonian fluids is now well-understood both physically and probabilistically. However, in complex fluids particles are expected to display non--classical, or \textit{anomalous}, behavior due to the viscoelasticity of the fluid. As in the early analysis of diffusion, biophysicists dedicate a great deal of attention to the average distance traveled by a particle, namely, the mean squared displacement $\mu_2(t) := E X^2(t) $ (MSD), where $X(t)$ denotes the position of the particle at instant $t$. For a given time window $I$, we can express the ``local'' MSD in the form
\begin{equation}\label{e:asympt_exp}
E X^2(t) \approx \theta t^\alpha, \quad \alpha, \theta > 0, \quad t \in I,
\end{equation}
where the parameters $\theta$ and $\alpha$ are called the diffusivity coefficient and diffusion exponent, respectively. The microparticle is said to be sub-, super- or simply diffusive if the $\alpha$ is less than, greater than, or equal to 1, respectively. When $\alpha \neq 1$, the diffusion is commonly named anomalous (see O'Malley and Cushman \cite{omalley:cushman:2012} for a different perspective). The interval $I$ in \eqref{e:asympt_exp} can be of finite length or open-ended, according to the demands of physical analysis. In the former case, \eqref{e:asympt_exp} expresses transient MSD behavior, as observed in polymer physics (Rubinstein and Colby \cite{rubinstein2003polymer}, Kremer and Grest \cite{kremer1990dynamics}). Alternatively, \eqref{e:asympt_exp} describes the asymptotic MSD behavior (see the relation \eqref{e:gamma=<h^2H_0} for the accurate mathematical depiction of \eqref{e:asympt_exp} in the context of this paper).

Statistical evidence of anomalous diffusion has turned up in several contexts, including biodiffusion (Valentine et al.\ \cite{valentine:kaplan:thota:crocker:gisler:prudhomme:beck:weitz:2001}), blinking quantum dots (Brokmann et al.\ \cite{brokmann:hermier:messin:desbiolles:bouchaud:dahan:2003}, Margolin and Barkai \cite{margolin:barkai:2005}) and fluorescence studies in single-protein molecules (Kou and Xie \cite{kou:xie:2004}, Kou \cite{kou:2008}). The dominant statistical technique in the biophysical literature for estimating the diffusion exponent $\alpha$ is what we will call the sample mean squared displacement ($\MSD$). Suppose that a microrheological experiment generates a tracer bead sample path with observations $X(\Delta j)$, $j=0,1,\hdots, n$, where $\Delta \in \bbN$ stands for the sampling rate. The pathwise statistic
\begin{equation}\label{e:MSD^}
\overline{\mu}_2( \Delta h  ) := \frac{1}{n-h}\sum_{j=1}^{n-h} \left(X(\Delta (j+h) )-X(\Delta j )  \right)^2
\end{equation}
is the $\MSD$ at $h$, i.e., the statistical counterpart of $\mu_2(\cdot)$ (for notational simplicity, we do not display the dependency of $\overline{\mu}_2$ on $n$). Under (\ref{e:asympt_exp}), and assuming stationary increments, for $m$ lag values $h_1< ... <h_m$ and $m \ll n$ one hopes for ergodicity, namely,
\begin{equation}\label{eq:pathwise-msd}
\overline{\mu}_2( \Delta h_k  ) \approx EX^2(\Delta h_k), \quad k=1,\hdots,m.
\end{equation}
One then generates $(\widehat{\log \theta}, \widehat{\alpha})$ by means of the linear regression
\begin{equation}\label{e:regression}
 \log \overline{\mu}_2(\Delta h_k)  = \log \theta + \alpha \log( \Delta h_k  ) +\varepsilon_k, \quad k=1,...,m,
\end{equation}
possibly over several independent particle paths, where $\{\varepsilon_k \}_{k = 1,\hdots,m}$ is a random vector with an unspecified distribution. Plots of $\MSD$ curves as a function of the lag $h$, sometimes on a log-log scale (see Figure \ref{f:msd_plots}), are widely reported as part of diffusion analysis (e.g., Valentine et al.\ \cite{valentine:kaplan:thota:crocker:gisler:prudhomme:beck:weitz:2001}, Suh et al.\ \cite{suh:dawson:hanes:2005}, Matsui et al \cite{matsui:wagner:hill:etal:2006}, Lai et al \cite{lai:wang:cone:wirtz:hanes:2009}, Lieleg et al.\ \cite{lieleg:vladescu:ribbeck:2010}). The choice of lags $h_1,\hdots,h_m$ reflects the analyst's visual perception of the range where the slope of the MSD curves stabilize and thus indicate the true diffusive regime and power law \eqref{e:asympt_exp}.

\begin{figure}
	\begin{center}
	\includegraphics[height=2.5in,width=3in]{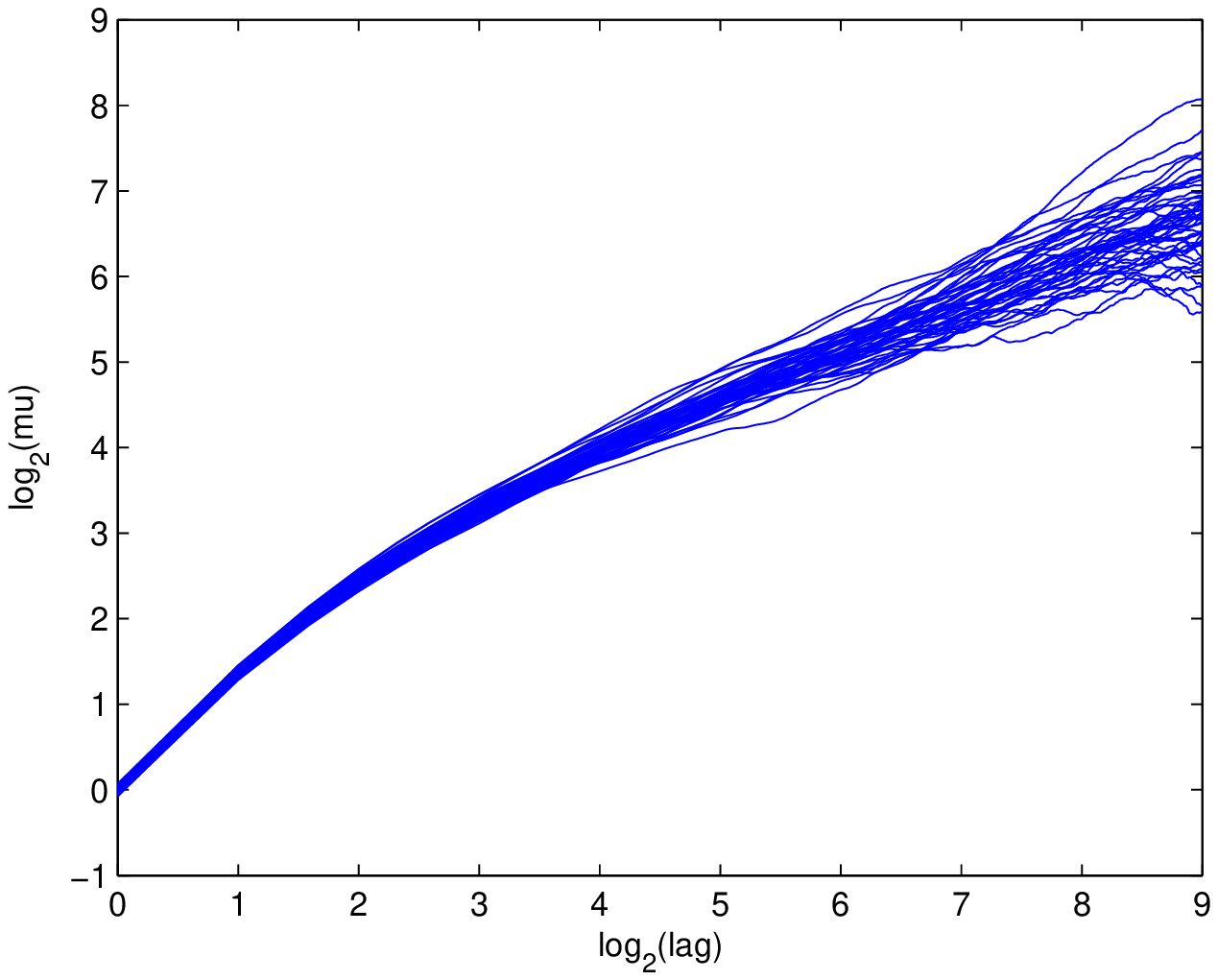} \ \includegraphics[height=2.5in,width=3in]{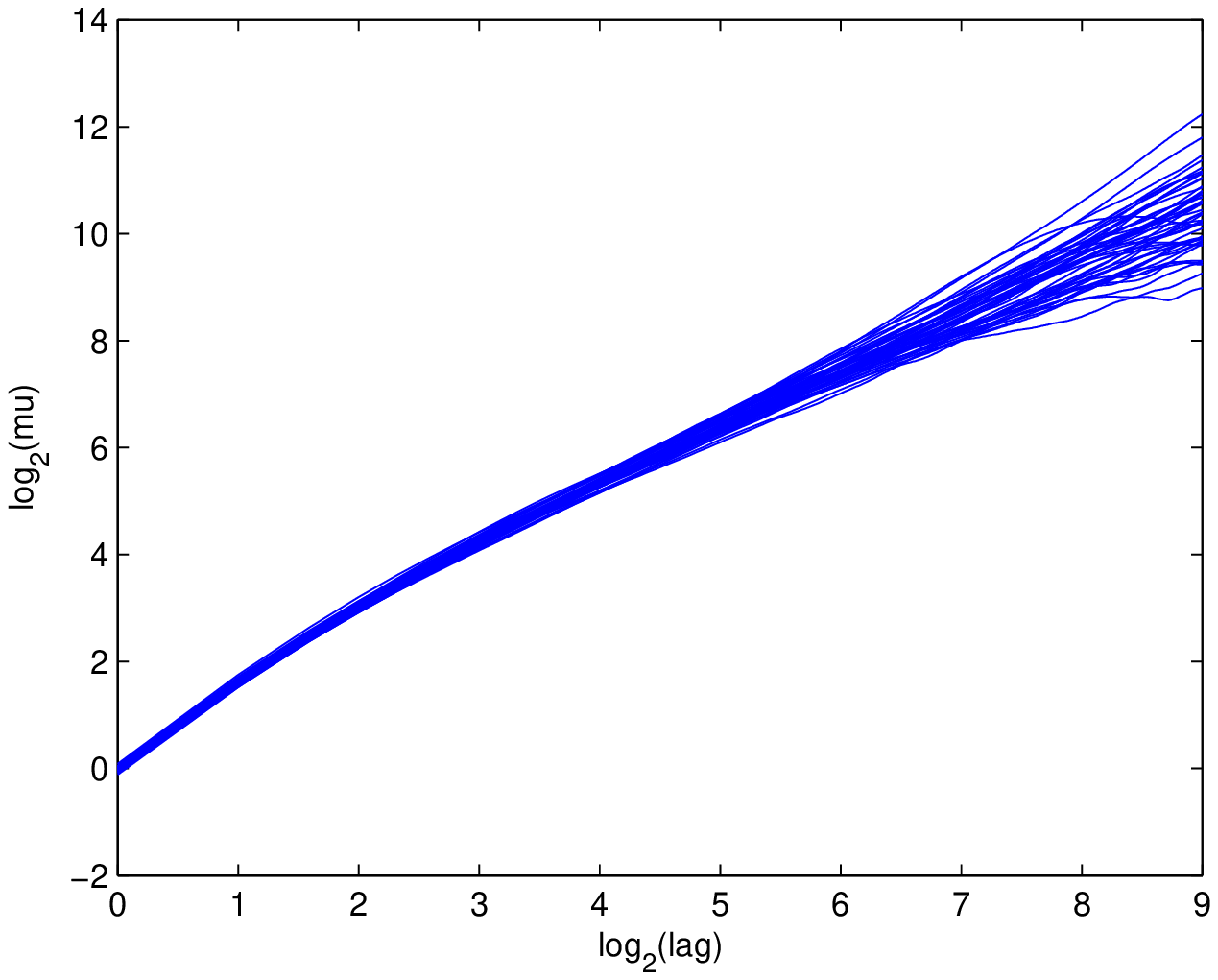}
	\caption{\label{f:msd_plots} $\MSD$ plots, each with 50 paths of size $2^{12}$ from an ifOU process with parameters $\lambda = 1$ and $H$ (see Definition \ref{def:ifOU}). Left: $H = 1/4$ (subdiffusive: $\alpha = 1/2$). Right: $H = 1/2$ (diffusive: $\alpha = 1$).
	}
	\end{center}
\end{figure}

The stochastic properties of the $\MSD$ depend on the underlying class of stochastic processes. In the review paper Meroz and Sokolov \cite{meroz:sokolov:2015}, the authors classify physical models for subdiffusive behavior according to whether one assumes the presence of binding-unbinding events, of geometrical constraints on the particle's movement, or the medium is viscoelastic. This leads to three popular families of stochastic processes, respectively, those of continuous time random walks (Metzler and Klafter \cite{metzler:klafter:2000}, Meerschaert and Scheffler \cite{meerschaert:scheffler:2004}), of random walks on fractals (Havlin and Ben-Avraham \cite{havlin:ben-avraham:1987}), and of the celebrated fractional Brownian motion (fBm; see Example \ref{ex:fBm}). In this paper, we focus on the latter family, more precisely, that of fractional, stationary increment processes (Barkai et al.\ \cite{barkai:garini:metzler:2012}, Lysy et al.\ \cite{lysy:pillai:hill:forest:mellnik:vasquez:mckinley:2014}).

The ergodicity of the $\MSD$ moments was established in Deng and Barkai \cite{deng:barkai:2009} for various families of fractional processes (see also Sokolov \cite{sokolov:2008}, Metzler et al.\ \cite{metzler:tejedor:jeon:he:deng:burov:barkai:2009}, Jeon and Metzler \cite{jeon:metzler:2010}, Burov et al.\ \cite{burov:jeon:metzler:barkai:2011}, Jeon et al.\ \cite{jeon:barkai:metzler:2013}, Sandev et al.\ \cite{sandev:metzler:tomovksi:2012}). Finite sample approximations to the distribution of the $\MSD$ under Gaussianity are provided in Grebenkov \cite{grebenkov:2011prob} (see also Qian et al.\ \cite{qian:sheetzL:elson:1991}, Grebenkov \cite{grebenkov:2011functionals}, Boyer et al.\ \cite{boyer:dean:mejia:oshanin:2012}, Andreanov and Grebenkov \cite{andreanov:grebenkov:2012}, Nandi et al.\ \cite{nandi:heinrich:lindner:2012}, Boyer et al.\ \cite{boyer:dean:mejia:oshanin:2013}). Nevertheless, so far $\MSD$-based analysis of tracking data has missed one essential feature of statistical methods, namely, the \textit{limiting distribution} of the random vector
\begin{equation}\label{e:MSD^_vector}
\Big(\overline{\mu}_2( \Delta h_1  ) ,\hdots, \overline{\mu}_2( \Delta h_m  )\Big ).
\end{equation}
The purpose of this paper is to fill this gap. We work under the assumption that the particle undergoes a Gaussian process whose stationary increments display a covariance function $\gamma$ satisfying a decay condition of the type
\begin{equation}\label{e:decay_cov_heuristic}
\gamma_{h}(z)_{1,2} \sim C z^{\alpha-2}h_{1}h_{2}, \quad z \rightarrow \infty,
\end{equation}
for some real constant $C$, where $h_{1}$, $h_{2}$ represent lag sizes (see the expressions \eqref{e:gamma_hk,hl} and \eqref{e:gamma=z2H-2_hk_hl_const+resid} for precise definitions, notation and statements). As in the particular case of fBm, such a particle is not constrained by boundaries (such as those found in a cell) or a potential.

We assume the availability of just one sample path. This models the situation in which the biophysical samples are physically \textit{heterogeneous} (Valentine et al.\ \cite{valentine:kaplan:thota:crocker:gisler:prudhomme:beck:weitz:2001}, Dawson et al.\ \cite{dawson:wirtz:hanes:2003}, Monnier et al.\ \cite{monnier:guo:mori:etal:2012}). Complex biomaterials such as mucus, or simulants such as agarose and hyaluronic acid, are expected to be heterogeneous due to the unequal distribution of chains of polymers.
Since the multiple $\MSD$ averages are formed from the same particle path, then even when $|h_2-h_1|$ is large the associated coefficients \eqref{e:MSD^} still display strong correlation (Monnier et al.\ \cite{monnier:guo:mori:etal:2012}). Our main result (Theorem \ref{t:MSD_asymptotic_dist}) shows that this yields limiting distributions and convergence rates that depend on the diffusion exponent range according to a familiar trichotomy in the literature on fractional processes. When $0 < \alpha \leq 3/2$, the asymptotic distribution is Gaussian, though the case $\alpha = 3/2$ demands a non--standard convergence rate. When $3/2 < \alpha < 2$ the convergence rate depends on the diffusion exponent and the asymptotic distribution is non--Gaussian; this reflects the classical results by M.\ Rosenblatt \cite{rosenblatt:1961} and M.\ Taqqu \cite{taqqu:1975}. This type of result is well-known for fixed sequences of Gaussian, stationary random variables, or for $p$-variations of shrinking interval size of Gaussian processes (Guyon and Le\'{o}n \cite{guyon:leon:1989}, Peltier and V\'{e}hel \cite{peltier:vehel:1994}, Hosking \cite{hosking:1996}, Bardet \cite{bardet:2000}, Buchmann and Chan \cite{buchmann:chan:2009}). By contrast, we consider the $\MSD$ statistics in the same format found in the biophysical literature, namely, we take the lag limit $h \rightarrow \infty$. Moreover, whereas the related literature on Hermitian processes and random fields often makes use of Wiener-It\^{o} chaos expansions and Malliavin calculus (e.g., Nourdin et al.\ \cite{nourdin:nualart:tudor:2010}, R\'{e}veillac et al.\ \cite{reveillac:stauch:tudor:2012}), in this work we develop our results in the style of Rosenblatt's classical arguments as to make the statements and techniques more readily available to the interested reader with a biophysical background. The asymptotic distributions provided allow for a new statistical perspective on the many numerical-experimental results reported by the biophysical community, and make it possible to mathematically compare $\MSD$-based analysis with that based on other candidate statistical techniques, e.g., in the Fourier and wavelet domains.

It should be stressed that we do not assume exact self-similarity (see relation \eqref{e:ss}). Dispensing with the latter property is important because it is often of interest to start from a Newtonian instance, such as the generalized Langevin equation (GLE; e.g., Lysy et al.\ \cite{lysy:pillai:hill:forest:mellnik:vasquez:mckinley:2014}, p.6), to arrive at an anomalous diffusion model that displays non-fractional short range behavior. In particular, we show that our results encompass the stationary-increment process induced by a fractional Ornstein-Uhlenbeck (fOU) velocity process (see Definition \ref{def:ifOU}). The latter can be regarded as a spectrally simplified model for fractional instances of the GLE (see \eqref{e:specdens_fOU}).

The paper is divided as follows. Section \ref{s:assumptions} contains most definitions and the assumptions used throughout the paper. We also shed light on the proposed assumptions by showing that they imply the properties \eqref{e:asympt_exp} and \eqref{e:decay_cov_heuristic}. In Section \ref{s:MSD_distribution}, we state and discuss weak limits for the $\MSD$. Furthermore, Monte Carlo experiments are used to illustrate the Gaussian or non-Gaussian nature of the $\MSD$ distribution, and to study the quality of the asymptotic approximation. All proofs can be found in the Appendix.

\section{Preliminaries and assumptions}\label{s:assumptions}

All through the paper, $C$ is used in bounds to denote a constant that does not depend on the sample size $n$, and which may change from one line to another. For two sequences of real numbers $\{a_n\}_{n \in \bbN}$, $\{b_n\}_{n \in \bbN}$, the expression $a_n \sim b_n$ means that $\frac{a_n}{b_n} \rightarrow 1$ as $n \rightarrow \infty$.

Recall that a stochastic process $X$ is said to have stationary increments when $\{X(t + h) - X(h)\}_{t \in \bbR}$ has the same finite-dimensional distributions for any time shift $h \in \bbR$. The stochastic process in \eqref{e:asympt_exp} is assumed to satisfy the following condition.\\

\medskip

\noindent {\sc Assumption (A1)}: $X = \{X(t)\}_{t \in \bbR}$ is a Gaussian, stationary-increment process with harmonizable representation
\begin{equation}\label{x_spec_rep}
X(t)=C_{\alpha} \int_{\bbR} \frac{e^{itx}-1}{ix}  \frac{s(x)}{|x|^{\alpha/2-1/2}} \widetilde{B}(dx),
\end{equation}
where $\alpha\in (0,2)$, $C_{\alpha} \neq 0$, $\widetilde{B}(dx)$ is a $\bbC$-valued Brownian measure such that $\widetilde{B}(-dx)=\overline{\widetilde{B}(dx)}$, $E\widetilde{B}(dx)\overline{\widetilde{B}(dx')}=0$ $(x\neq x')$, and $E |\widetilde{B}(dx)|^2=dx$. The function $s(x)$ is a bounded and complex-valued function with $\abs{s(0)}^2 = 1$, and
\begin{equation}\label{e:s(x)}
|\abs{s(x)}^2-1|\leq C_0 x^{\delta_0},\quad x\in (-\varepsilon_0,\varepsilon_0),
\end{equation}
for constants $C_0,\delta_0,\varepsilon_0>0$.

\medskip

\noindent In particular, the representation \eqref{x_spec_rep} implies that $EX(t) = 0$, $t \in \bbR$. However, this is inconsequential for modeling, since one can always assume that a single diffusing particle starts at zero. In turn, the condition \eqref{e:s(x)} is mild (c.f.\ Moulines et al.\ \cite{moulines:roueff:taqqu:2007fractals}, p.302, relation (4)) and plays a technical role in the proof of Proposition \ref{prop:assumption_regu_shortmemo} below.

\begin{example}\label{ex:fBm}
FBm is the only Gaussian, self-similar, stationary-increment process (Taqqu \cite{taqqu:2003}, Proposition 2.3). The self-similarity of the fBm $B_H = \{B_H(t)\}_{t \in \bbR}$ means that, for a Hurst parameter $0 < H \leq 1$, the scaling relation
\begin{equation}\label{e:ss}
\{B_{H}(ct)\}_{t \in \bbR} \stackrel{{\mathcal L}}= \{c^{H} B_{H}(t)\}_{t \in \bbR}, \quad c > 0,
\end{equation}
is satisfied. FBm has mean zero, and by Gaussianity, it is characterized by its closed-form covariance function
\begin{equation}\label{e:fBm_cov}
EB_{H}(s)B_H(t) = \frac{\sigma^2}{2} \{|t|^{2H} + |s|^{2H} - |t-s|^{2H}\}, \quad s,t \in \bbR.
\end{equation}
In particular, when $\sigma^2 = 1$, we call $B_H$ a standard fBm. Moreover, by taking $s = t$ in \eqref{e:fBm_cov}, the expression \eqref{e:asympt_exp} holds at all $t$ as an equality with
\begin{equation}\label{e:alpha}
\alpha = 2H.
\end{equation}
When $H$ is less than, greater than or equal to 1/2, fBm is sub-, super- or simply diffusive (Brownian motion), respectively. The harmonizable representation of a standard fBm is given by
\begin{equation}\label{e:standard_fBm_harmonizable}
X(t)= C_H \int_{\bbR} \frac{e^{itx}-1}{ix} \frac{1}{|x|^{H-1/2}} \widetilde{B}(dx),
\end{equation}
where
\begin{equation}\label{e:CH_standard_fBm}
C_H = \sqrt{\pi^{-1} H \Gamma(2H)\sin(H\pi)}
\end{equation}
(see Taqqu \cite{taqqu:2003}, p.31, expression (9.8)). Thus, fBm satisfies (A1) with $C_\alpha = C_H$ and $s(x)=1$.
\end{example}

Let $\Delta = 1$ in \eqref{e:MSD^}. We assume that an experiment produced \textit{one }sequential sample $X_1,\hdots,X_n$, $n \in \bbN$, of observations from the stochastic process $X$. For $h\in\bbN$, let
\begin{equation}\label{e:hk}
\{h_k:=w_kh\}_{k=1,...,m}
\end{equation}
be distinct integer-valued increment sizes, where $h_m\leq n-1$, $w_1 < w_2 < \hdots < w_m$ and $m\leq n-1$. We can define an associated vector of increments
\begin{equation}\label{e:Y_increments}
{\mathbf Y}=(Y_1(h_1),\hdots,Y_{n-h_1}(h_1);Y_1(h_2),\hdots,Y_{n-h_2}(h_2);\hdots;Y_1(h_m),\hdots,Y_{n-h_m}(h_m))^T,
\end{equation}
where
$$
Y_i(h_k) := X(i + h_k) - X(i), \quad h_k \in \bbN \cup \{0\}, \quad k = 1,\hdots,m.
$$
Since $X$ is a stationary-increment process, then the cross product
$$
EY_{j+z}(h_{k_1})Y_j(h_{k_2})=E(X(j+z + h_{k_1}) - X(j+z))(X(j + h_{k_2}) - X(j))
$$
$$
=E(X(z + h_{k_1}) - X(z))(X(h_{k_2}) - X(0))=E Y_{z}(h_{k_1})Y_{0}(h_{k_2})
$$
is not a function of $j$. Denote the covariance matrix of the increments by $\Sigma_{{\mathbf Y}}(h_1,\hdots,h_m)$, where an entry has the form
\begin{equation}\label{e:gamma_hk,hl}
\gamma_{h}(z)_{k_1,k_2}=EY_{j+z}(h_{k_1})Y_j(h_{k_2}),
\end{equation}
for $k_1,k_2=1,\hdots,m$. Note that \eqref{e:gamma_hk,hl} satisfies the symmetry relation
$$
\gamma_{h}(-z+ h_{k_2} - h_{k_1})_{k_1,k_2} = \gamma_{h}(z)_{k_1,k_2}, \quad z \in \bbZ.
$$
The self-similarity of fBm (see \eqref{e:ss}) makes the asymptotic distribution of the $\MSD$ much simpler to establish (see Peltier and V\'{e}hel \cite{peltier:vehel:1994}, Proposition 4.2). Since we do not assume self-similarity, as in the biophysical literature we need to make the size of the lags themselves go to infinity, though slower than the sample size $n$. This mathematically expresses what biophysicists do in practice: $h$ has to be large enough for the $\MSD$ regime to become linear, but at the same time cannot be too large because of the increased variance of the $\MSD$. This is illustrated in Figure \ref{f:msd_plots} and accurately described in assumption (A2), stated next.

\medskip

\noindent {\sc Assumption (A2)}: for $h=h(n) \in \bbN \cup \{0\}$, $n\in\bbN$,
\begin{equation}\label{e:h(n)}
\frac{h(n) \log^2(n)}{n}+\frac{n}{h(n)^{1+\delta/2}}\rightarrow 0,\quad n\rightarrow \infty,\\
\end{equation}
where
\begin{equation}\label{e:delta}
\delta=\min(\alpha/2,\delta_0/2)
\end{equation}
(see \eqref{x_spec_rep} and \eqref{e:s(x)} for the definitions of $\alpha$ and $\delta_0$).
\medskip

As anticipated in the Introduction, the expressions \eqref{e:gamma=<h^2H_0} and \eqref{e:gamma=z2H-2_hk_hl_const+resid} in the following proposition give exact mathematical meaning to the heuristic properties \eqref{e:asympt_exp} and \eqref{e:decay_cov_heuristic}.

\begin{proposition}\label{prop:assumption_regu_shortmemo}
Suppose the assumptions (A1) and (A2) hold.
\begin{itemize}
  \item [(i)] Then, there is a constant $\theta >0$ such that
\begin{equation}\label{e:gamma=<h^2H_0}
 \abs{\frac{E X^2(h)}{\theta h^{\alpha}}-1} \leq Ch^{-\delta},\quad n\rightarrow \infty,
\end{equation}
where $\delta>0$ is given by \eqref{e:delta};
  \item [(ii)] moreover, for any $k_1,k_2=1,...,m$ and $\gamma_{h}(z)_{k_1,k_2}$ as in \eqref{e:gamma_hk,hl}, \begin{equation}\label{e:gamma=z2H-2_hk_hl_const+resid}
\gamma_{h}(z)_{k_1,k_2}=|z|^{\alpha-2}h^2w_{k_1} w_{k_2} \{\tau + g(z,h )_{k_1,k_2} \}, \quad n\rightarrow \infty,
\end{equation}
\begin{equation}\label{e:hm+1<=z<=n}
\quad h_{m} + 1\leq |z|\leq n,
\end{equation}
where
\begin{equation}\label{e:tau}
\tau=\tau(\alpha)=\Big( \frac{C_\alpha}{C_H}\Big)^2 \frac{\alpha(\alpha-1)}{2},
\end{equation}
$C_H$ is given by \eqref{e:CH_standard_fBm}, and the residual function $g(\cdot,\cdot)_{k_1,k_2}$ satisfies
\begin{equation}\label{e:residual_function_g}
|g(z,h)_{k_1,k_2}|\leq C \Big( \frac{h}{|z|} \Big)^{\delta},\quad 0 < h \leq |z|.
\end{equation}
\end{itemize}
\end{proposition}

Besides fBm, another model for anomalous diffusion used in this work is what we call the integrated fractional Ornstein-Uhlenbeck (ifOU) process. One major difference between fBm and the ifOU process is that the latter is not exactly self-similar. The ifOU is an example of a fractional process whose $\MSD$ asymptotics are naturally studied under the assumption (A2).

To define the ifOU process, recall that the fractional Ornstein-Uhlenbeck process (fOU) is the a.s.\ continuous solution to the fBm-driven Langevin equation
\begin{equation} \label{e:fOU_SDE}
dV(t) = - \lambda V(t) dt + \sigma dB_{H}(t), \quad t \geq 0,\quad \lambda>0,\quad \quad 0 < H < 1
\end{equation}
(Rao \cite{prakasarao:2010}, p.78). The a.s.\ continuous process
$$
V(t) = \sigma \int^{t}_{-\infty} e^{- \lambda(t-u)}dB_H(u), \quad t > 0,
$$
solves \eqref{e:fOU_SDE} with initial condition $V(0)$ defined by the same integral. When $H = 1/2$, the solution is the classical Ornstein-Uhlenbeck process. We are interested in the stationary-increment counterpart of the fOU process, as put forward in the next definition.

\begin{definition}\label{def:ifOU}
Given a fOU process $\{V(t)\}_{t \geq 0}$, the associated ifOU process is given by
\begin{equation}\label{e:ifOU_def}
X(t) = \int^{t}_{0}V(s) ds, \quad t > 0.
\end{equation}
The integrand in \eqref{e:ifOU_def} is a version of $V$ with continuous paths (see Didier and Fricks \cite{didier:fricks:2014}, p.719, Lemma A.4).
\end{definition}

The ifOU is a simple parametric model for anomalous diffusion. This can be seen in the Fourier domain, based on the harmonizable representation
\begin{equation}\label{eqn:spec_int_fOU}
X(t) = \sigma \sqrt{\Gamma(2H+1) \sin(\pi H)} \int_{\bbR} \frac{e^{itx}-1}{ix} \frac{1}{\lambda + i x} \frac{1}{|x|^{H-1/2}} \widetilde{B}(dx).
\end{equation}
The spectral density
\begin{equation}\label{e:specdens_fOU}
f_X(x) = \sigma^2 \Gamma(2H+1) \sin(\pi H) \frac{1}{\lambda^2 + x^2} \frac{1}{|x|^{2H-1}}, \quad x \in \bbR \backslash\{0\},
\end{equation}
exhibits the short range dependence term $(\lambda^2 + x^2)^{-1}$ besides the fractional term $|x|^{1-2H}$, with a tuning parameter $\lambda$.

Let $H \in (0,1)\backslash\{1/2\}$. By \eqref{e:specdens_fOU} and dominated convergence, the covariance function $\gamma(s) =  \textnormal{Cov}(V(t),V(t+s))$ of $V$ is continuous. Furthermore, in Cheridito et al.\ \cite{cheridito:kawaguchi:maejima:2003}, p.8, it is shown that
\begin{equation}\label{e:cov_fOU_velocity}
\gamma(s) = \frac{\sigma^2}{2} \sum^{N}_{n=1} \lambda^{-2n} \Big( \prod^{2n-1}_{k=0}(2H-k)\Big)s^{2H-2n} + O(s^{2H-2N-2}),
\end{equation}
for an arbitrary $N \in \bbN$, where the remainder is taken with respect to $s \rightarrow \infty$. Therefore,
\begin{equation}\label{e:|E X(t_1)X(t_2)|}
|E X(t_1)X(t_2)| \leq \int^{t_1}_{0} \int^{t_2}_{0} |\gamma(s_1 - s_2)| ds_1 ds_2 < \infty,
\end{equation}
whence the integral \eqref{e:ifOU_def} is also well-defined in the mean squared sense (see Cram\'{e}r and Leadbetter \cite{cramer:leadbetter:1967}, p.86). By \eqref{eqn:spec_int_fOU}, like fBm the ifOU also satisfies (A1) and the conclusions of Proposition \ref{prop:assumption_regu_shortmemo} apply, where the relation between $\alpha$ and $H$ is again given by \eqref{e:alpha}. Note that when the ifOU is simply diffusive ($H = 1/2$, or $\alpha = 1$), the expression \eqref{e:gamma=z2H-2_hk_hl_const+resid} holds with $\tau = 0$.

\section{The asymptotic distribution of MSD--based anomalous diffusion parameter estimators}\label{s:MSD_distribution}

The following theorem is the main result of this paper. It gives the asymptotic distribution of a random vector of $\MSD$ entries at different lag values, according to subranges of the diffusion exponent $\alpha$. For the theorem, recall that in the notation \eqref{e:Y_increments}, the $\MSD$ at a given lag value is given by
\begin{equation}\label{e:mu2(hk)}
\overline{\mu}_2(h_k) = \frac{1}{N_k}\sum^{N_k}_{i=1}Y^{2}_i(h_k),
\end{equation}
where
\begin{equation}\label{e:Nk}
N_k := n- h_k, \quad k = 1, \hdots,m.
\end{equation}

\begin{theorem}\label{t:MSD_asymptotic_dist}
Suppose the assumptions (A1) and (A2) hold. Let $m \in \bbN$ be the chosen number of $\MSD$ lag values, and let $N_k$ be as in \eqref{e:Nk}, $k = 1,\hdots,m$. Then,
\begin{equation}\label{e:MSD_asymptotic_dist}
\bigg( \frac{N_k}{\eta(N_k)\zeta(h_k)}(\overline{\mu}_2(h_k)-EX^2(h_k))\bigg)_{k=1,\hdots,m}\stackrel{d}{\rightarrow} {\mathbf{Z}}, \quad n \rightarrow \infty,
\end{equation}
where
\begin{equation}\label{e:eta_zeta}
\left\{\begin{array}{ccc}
0 < \alpha < 3/2: & \eta(n) = \sqrt{n},\, \zeta(h)=h^{\alpha+1/2} ;\\
\alpha = 3/2: & \eta(n) = \sqrt{n\log(n)},\, \zeta(h)=h^{2};\\
3/2 < \alpha < 2: & \eta(n) = n^{\alpha-1},\, \zeta(h)=h^{2}.
\end{array}\right.
\end{equation}
In \eqref{e:MSD_asymptotic_dist},
\begin{itemize}
\item [(i)] if $0 < \alpha < 3/2$, then ${\mathbf{Z}} \sim N(0,\Sigma)$, where the entry $k_1,k_2$ of the matrix $\Sigma = \Sigma(\alpha)$ is given by
\begin{equation}\label{e:Sigma_0<H<3/4}
\Sigma_{k_1, k_2} = 2 w_{k_1}^{-\alpha-1/2}w_{k_2}^{-\alpha-1/2} \Big(\frac{C_\alpha}{C_H}\Big)^4 \norm{\widehat{G}(y;w_{k_1},w_{k_2})}^2_{L^2(\bbR)}, \quad k_1,k_2 = 1,\hdots,m,
\end{equation}
\begin{equation}\label{e:G_hat}
\widehat{G}(y;w_{k_1},w_{k_2})=C^2_H \frac{(e^{iw_{k_1} y}-1)(e^{-iw_{k_2}y}-1)}{\abs{y}^{\alpha+1}},
\end{equation}
and $C_H$ is given in \eqref{e:CH_standard_fBm};
\item [(ii)] if $\alpha = 3/2$, then ${\mathbf{Z}} \sim N(0,\Sigma)$, where the entry $k_1,k_2$ of the matrix $\Sigma= \Sigma(\alpha)$ is given by
\begin{equation}\label{e:Sigma_H=3/4}
\Sigma_{k_1,k_2}=4\tau^2, \quad k_1,k_2 = 1,\hdots,m.
\end{equation}

\item [(iii)] if $3/2 < \alpha < 2$, $\mathbf{Z}$ follows a multivariate Rosenblatt-type distribution whose characteristic function is
\begin{equation}\label{e:3/4<H<1_multivariate_Rosen_limit}
    \phi_{\mathbf{Z}}({\mathbf{t}})=
        \exp\bigg\{\frac{1}{2}\sum_{s=2}^{\infty} \frac{[ 2i \tau \hspace{0.5mm}\sum^{m}_{k=1}t_k ]^s}{s}\hspace{1mm} c_s\bigg\},
\end{equation}
where, for $s \geq 2$,
\begin{equation}\label{e:cs}
    c_s = c_s(\alpha) = \int_0^1 \int_0^1 \cdots\int_0^1 |x_1-x_2|^{\alpha-2}|x_2-x_3|^{\alpha-2}\cdots|x_s-x_1|^{\alpha-2}dx_1dx_2\cdots dx_s.
\end{equation}
\end{itemize}
\end{theorem}

\begin{remark}\label{r:cs_is_finite}
It is worthwhile recalling the fact that the constant $c_s$ in \eqref{e:cs} is, indeed, finite. Indeed, by an application of the Cauchy-Schwarz inequality,
\begin{equation}\label{e:cs_finite}
c_s \leq \bigg( \int_0^1\int_0^1 \abs{x_1-x_2}^{2\alpha-4} dx_1 dx_2  \bigg)^{s/2}=\bigg(\frac{1}{(2\alpha-3)(\alpha-2)}\bigg)^{s/2}.
\end{equation}
\end{remark}

Theorem \ref{t:MSD_asymptotic_dist} allows us to develop the asymptotic distribution of the $\MSD$-based least square estimator of the diffusivity coefficient and diffusion exponent. Recast the (pathwise) system \eqref{e:regression} as the regression model
\begin{equation}\label{e:MSD_regression}
Q_n = M_n \boldsymbol\beta + {\boldsymbol \varepsilon}_n,
\end{equation}
where
\begin{equation}\label{e:lm_notation}
Q_n = \left(\begin{array}{c}
\log \overline{\mu}_2(h_1) \\
\vdots\\
\log \overline{\mu}_2(h_m)
\end{array}\right), \quad
{\boldsymbol \beta} =
\left(\begin{array}{c}
\log \theta \\
\alpha
\end{array}\right), \quad M_n = \left(\begin{array}{cc}
1 & \log(h_1)\\
\vdots & \vdots \\
1 & \log(h_m)
\end{array}\right),
\end{equation}
and ${\boldsymbol \varepsilon}$ has a distribution to be determined. We will denote by
\begin{equation}\label{e:lm_estimator}
\widehat{\boldsymbol \beta}_n:=(M_n^T M_n)^{-1}M_n^T Q_n = (\hspace{0.5mm}(\widehat{\log\theta})_n, \hspace{1mm}\widehat{\alpha}_n \hspace{0.5mm})^T
\end{equation}
the estimator generated by the ordinary least squares solution to the system \eqref{e:MSD_regression}. The next corollary describes the asymptotic distribution of the least squares estimator \eqref{e:lm_estimator}.

\begin{corollary}\label{t:MSD_asymptotic_coro}Suppose the assumptions (A1) and (A2) hold. Then, as $n\rightarrow\infty$,
    \begin{equation}\label{e:lm_estimator_asymptotic}
    \frac{n h^{\alpha}}{\eta(n)\zeta(h)}
    \left(
      \begin{array}{c}
        \frac{1}{\log h }((\widehat{\log\theta})_n-{\log \theta}) \\
        \widehat{\alpha}_n-{\alpha} \\
      \end{array}
    \right)
    \stackrel{d}{\rightarrow}
    \left(
      \begin{array}{c}
        U^T \\
        -U^T \\
      \end{array}
    \right)
    A{\mathbf{Z}}.
    \end{equation}
In \eqref{e:lm_estimator_asymptotic},
\begin{equation}\label{e:A(theta,alpha)}
    A =A(\theta,\alpha)=\diag( \zeta(w_1)/(\theta w_1^{\alpha}),...,{\zeta(w_m)}/(\theta w_m^{\alpha})),
    \end{equation}
$\eta(\cdot)$, $\zeta(\cdot)$ and ${\mathbf{Z}}$ are as in Theorem \ref{t:MSD_asymptotic_dist}, and
    \begin{equation}\label{e:u_and_v}
    U^T=\frac{1}{c_w}\Big(\sum_{k=1}^{m} \log(w_k/w_1) , ... , \sum_{k=1}^{m} \log(w_k/w_m)\Big)
    \end{equation}
    with constant
    \begin{equation}\label{e:c_det_of_MTM}
    c_w=m \sum_{k=1}^{m} \log^2(w_k) - \Big(\sum_{k=1}^{m} \log(w_k)\Big)^2.
    \end{equation}

\end{corollary}

Corollary \ref{t:MSD_asymptotic_coro} states that the limiting distributions for $\widehat{{\boldsymbol\beta}}_n$ are qualitatively distinct as a function of the underlying diffusion exponent $\alpha$. In particular, a non-Gaussian limit appears in the superdiffusive range $3/2<\alpha<2$. Though probably of little interest in the modeling of viscoelastic diffusion, superdiffusion appears in many other applications (e.g., Brokmann et al.\ \cite{brokmann:hermier:messin:desbiolles:bouchaud:dahan:2003}, Margolin and Barkai \cite{margolin:barkai:2005}; note that in these papers the processes are viewed as following L\'{e}vy walk-type dynamics).

\begin{remark}
Corollary \ref{t:MSD_asymptotic_coro} can be directly used in the construction of confidence intervals, at least starting from knowledge that $\alpha$ lies in one of the subregions $(0,3/2)$ or $(3/2,2)$ of the parameter space. To fix ideas, consider the parameter $\alpha$; the ensuing argument can be easily adapted for $\theta$. By Corollary \ref{t:MSD_asymptotic_coro}, $\widehat{\alpha}_n-{\alpha}$ is asymptotically equivalent to
$$
 -\frac{\eta(n)\zeta(h)}{n h^{\alpha}}U^T A{\mathbf{Z}}=: \frac{\eta(n)\zeta(h)}{n h^{\alpha}}\sigma(\theta,\alpha) Z(\alpha),
$$
where $\sigma(\theta,\alpha)$ is a smooth function of $(\theta,\alpha)$ defined as to make $Z(\alpha)$ a standardized random variable. When $0 < \alpha < 3/2$, this is clearly possible, since the limiting distribution is Gaussian and $-U^T A$ and $\Sigma$ are smooth functions of $\theta$ and $\alpha$ (see \eqref{e:Sigma_0<H<3/4}, \eqref{e:MSD_asymptotic_various_lags} and \eqref{e:u_and_v}). When $3/2<\alpha<2$, first note that ${\mathbf{Z}}$ in Theorem \ref{t:MSD_asymptotic_dist} is a rank 1 random vector. Indeed, recall that the characteristic function of a standardized (mean zero, variance one) Rosenblatt random variable is given by
\begin{equation}\label{e:standard_Rosenblatt}
\phi_{\alpha}(t) = \exp\Big\{ \frac{1}{2}\sum^{\infty}_{s=2}(2 i t \psi(\alpha))^s \frac{c_s}{s}\Big\}, \quad \psi(\alpha) := \Big(\frac{(2 \alpha - 3)(\alpha-1)}{2}\Big)^{1/2},
\end{equation}
(see Veillette and Taqqu \cite{veillette:taqqu:2013}, expression (4)). Let $Z_1$ be a Rosenblatt random variable with normalizing constant $\tau$ (i.e., with the latter in place of $\psi(\alpha)$ in \eqref{e:standard_Rosenblatt}), and let $Z_2 = \hdots = Z_m := Z_1$. Then, $Ee^{i \sum^{m}_{k=1}t_k Z_k} = Ee^{i Z_1 \sum^{m}_{k=1}t_k } = \exp\{\frac{1}{2}\sum_{s=2}^{\infty} [ 2i \tau \hspace{0.5mm}\sum^{m}_{k=1}t_k ]^s\frac{c_s}{s}\}$. So, denote by $\widetilde{Z}$ the limiting Rosenblatt random variable obtained in \eqref{e:tau}. Then, $Z_{\alpha}:= (\psi(\alpha)/\tau) \widetilde{Z}$ is standardized and, by \eqref{e:tau} and \eqref{e:standard_Rosenblatt}, the coefficient $\psi(\alpha)/\tau$ also depends smoothly on $\alpha$.

So, the consistency of $\widehat{\theta}_n$ and $\widehat{\alpha}_n$ for $\theta$ and $\alpha$, respectively, implies that of $\sigma(\widehat{\theta}_n,\widehat{\alpha}_n)$ for $\sigma(\theta,\alpha)$. When $0<\alpha < 3/2$, $Z(\alpha) \sim N(0,1)$, which is independent of $\alpha$. Then, an approximate $100(1-\xi)\%$ confidence interval for $\alpha$ is simply
 $$
 \Big(\widehat{\alpha}_n + \Big(\frac{h}{n}\Big)^{1/2}\sigma(\widehat{\theta}_n,\widehat{\alpha}_n) z_{\xi/2}, \widehat{\alpha}_n + \Big(\frac{h}{n}\Big)^{1/2}\sigma(\widehat{\theta}_n,\widehat{\alpha}_n) z_{1-\xi/2}\Big).
 $$
When $3/2 < \alpha < 2$, by \eqref{e:cs_finite} and the dominated convergence theorem, the characteristic function $\phi_{\alpha}(t) =: \phi(t,\alpha)$ of $Z_{\alpha}$ is continuous with respect to $\alpha$ for $t$ around the origin. Now consider the function $\phi(z,\alpha)$ with domain in $z$ extended to a vicinity of the origin of $\bbC$. By applying Theorem 7.1.1 in Lukacs \cite{lukacs1970} and the uniqueness of analytic continuation, we obtain that $\phi(t,\alpha)$ is continuous with respect to $\alpha$ for all $t\in \bbR$. Consequently, the cumulative distribution function and the quantile function $z_{\varsigma}(\alpha)$ are also continuous with respect to $\alpha$, whence $z_{\varsigma}(\widehat{\alpha}_n) \overset{P}{\rightarrow} z_{\varsigma}(\alpha)$ for $\varsigma \in (0,1)$. So, an approximate $100(1-\xi)\%$ confidence interval for $\alpha$ is
\begin{equation}\label{e:CI_Rosenblatt}
 \Big(\widehat{\alpha}_n + \Big(\frac{h}{n}\Big)^{2-\widehat{\alpha}_n}\sigma(\widehat{\theta}_n,\widehat{\alpha}_n) z_{\xi/2}(\widehat{\alpha}_n), \widehat{\alpha}_n + \Big(\frac{h}{n}\Big)^{2-\widehat{\alpha}_n}\sigma(\widehat{\theta}_n,\widehat{\alpha}_n) z_{1-\xi/2}(\widehat{\alpha}_n) \Big)
\end{equation}
(see Veillette and Taqqu \cite{veillette:taqqu:2013} for numerical results on the quantiles of the Rosenblatt distribution). In \eqref{e:CI_Rosenblatt}, we are using the fact $(\frac{h}{n})^{2-\alpha}/(\frac{h}{n})^{2-\widehat{\alpha}} \rightarrow 1$, which can be verified by taking logs.
\end{remark}

\begin{figure}[h]
	\begin{center}
	\includegraphics[height=2in,width=2.5in]{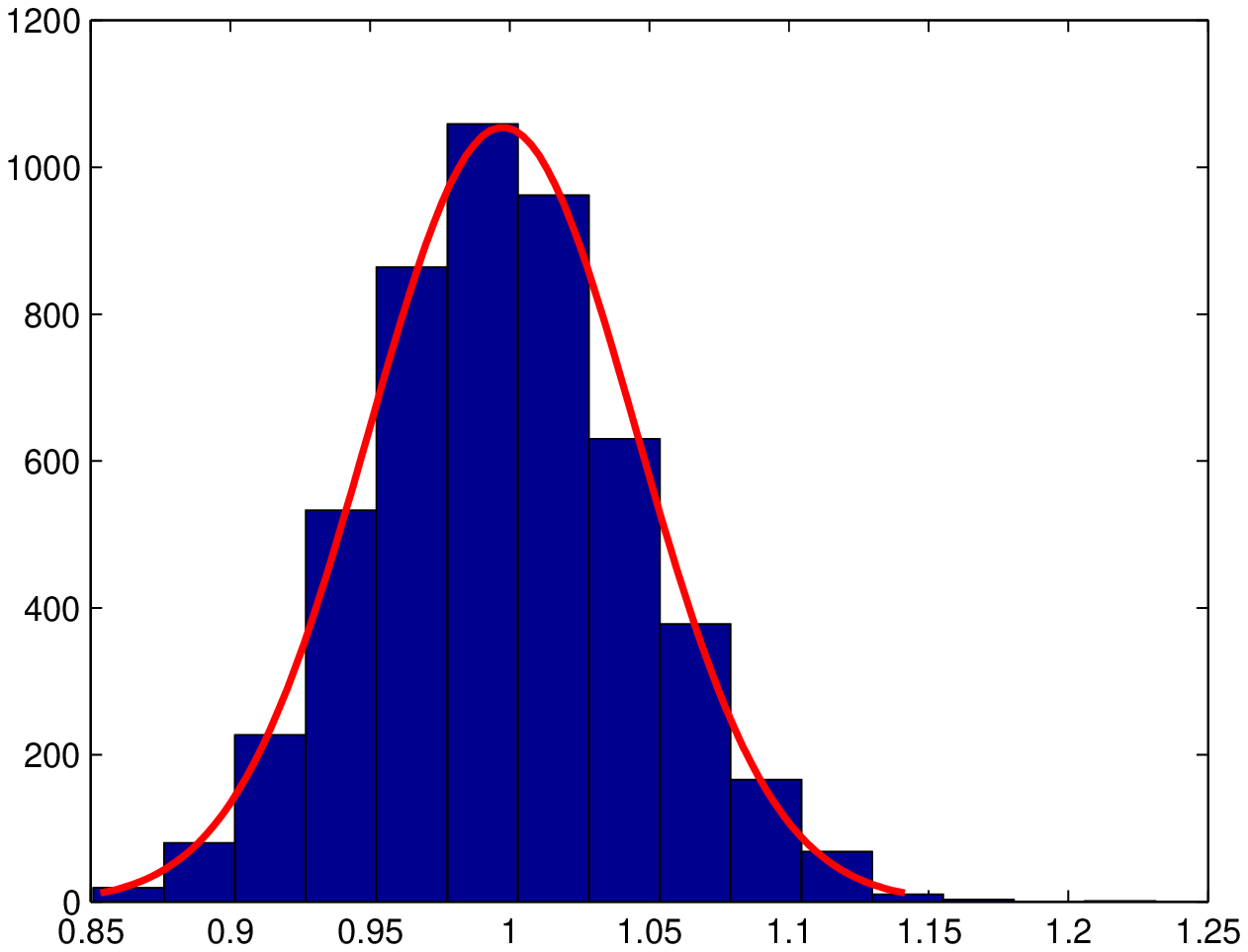} \ \includegraphics[height=2in,width=2.5in]{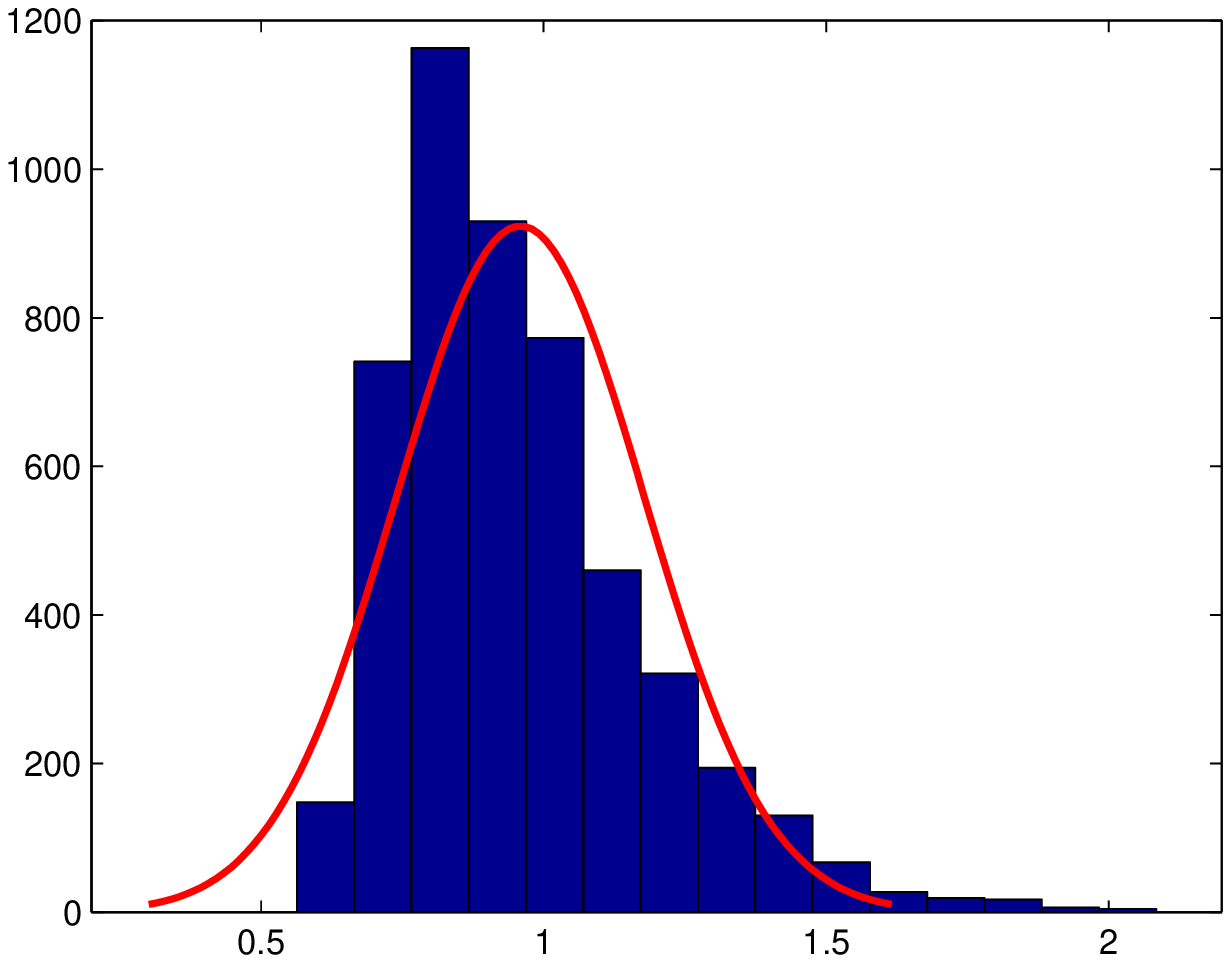}
	\includegraphics[height=2in,width=2.5in]{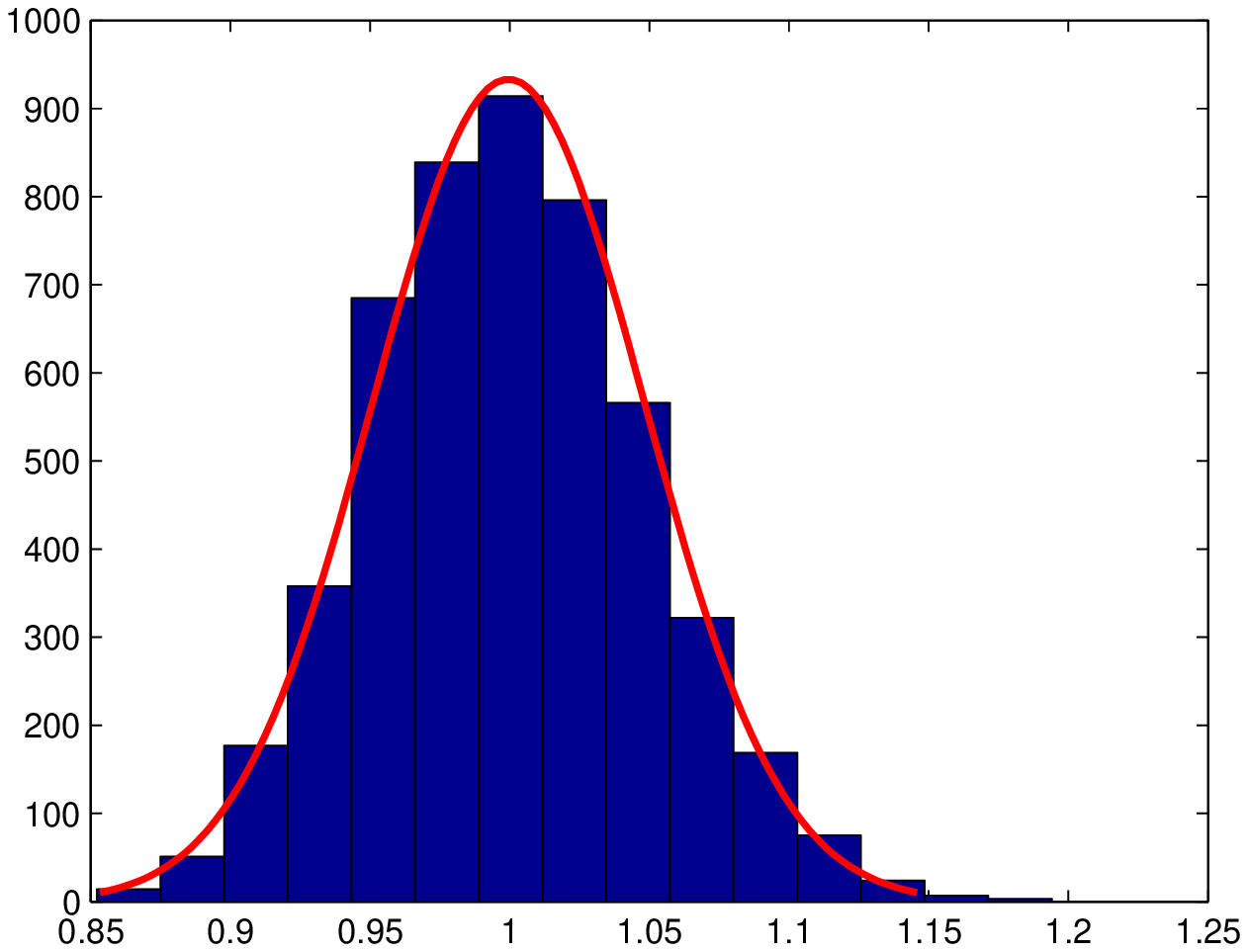} \ \includegraphics[height=2in,width=2.5in]{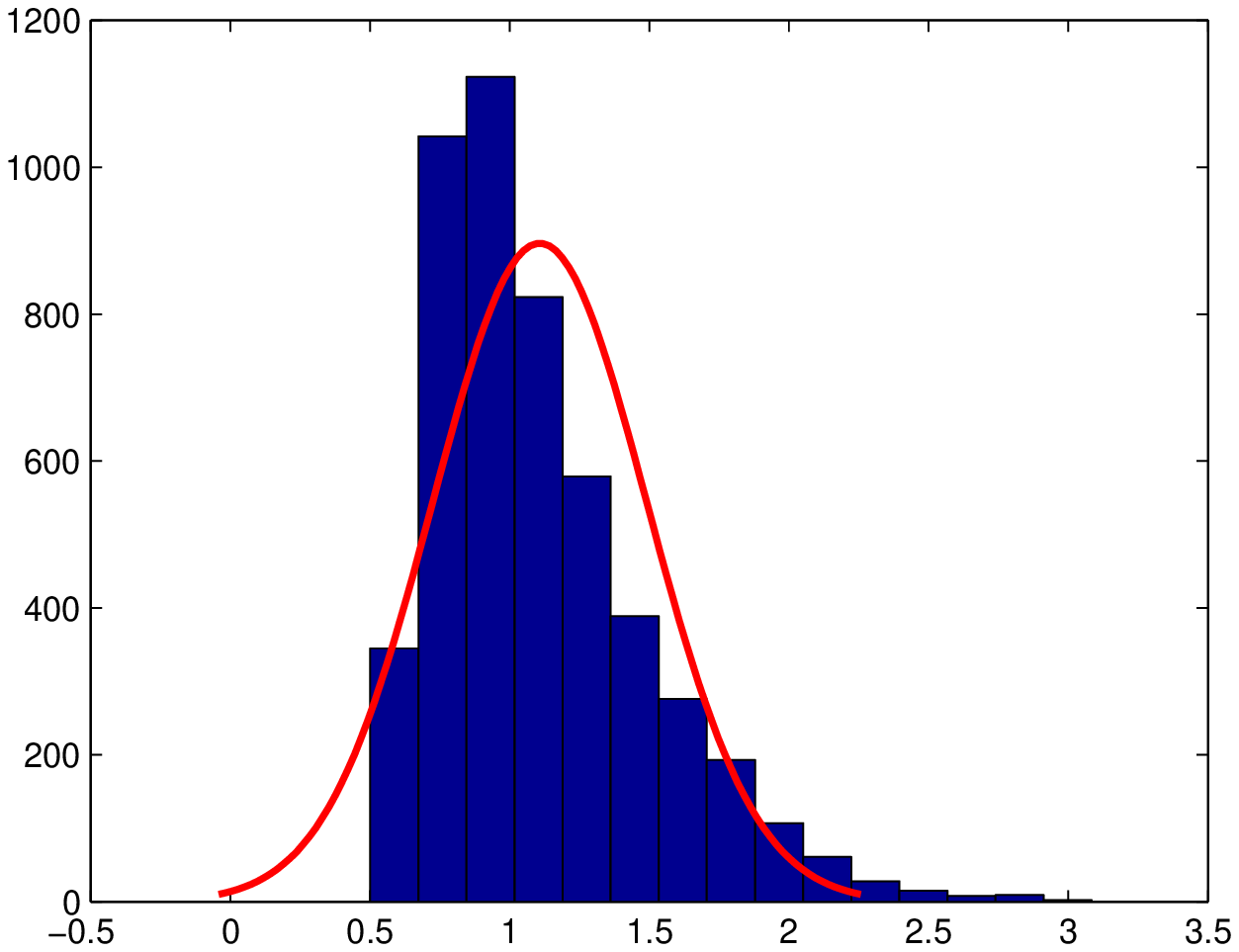}
	\caption{\label{f:mc_distribution_msd_fBm_ifOU} Monte Carlo distribution of $\overline{\mu}_{2}(1)$ over 5,000 paths of size $2^{10}$. Top plots, fBm (left: $\alpha = 0.5$; right: $\alpha = 1.8$. Bottom plots, ifOU (left: $\alpha = 0.5$; right: $\alpha = 1.8$. In both cases, $\lambda=1$).
	}
	\end{center}
\end{figure}

To study the finite-sample properties of $\MSD$--based estimation and the quality of the asymptotic approximations described in Theorem \ref{t:MSD_asymptotic_dist} and Corollary \ref{t:MSD_asymptotic_coro}, Monte Carlo experiments were conducted based on sub- and superdiffusive instances of fBm and ifOU processes. Figure \ref{f:mc_distribution_msd_fBm_ifOU} displays the Monte Carlo distributions of the $\MSD$, i.e., histograms and best Gaussian fit. The plots reflect the results described in Theorem \ref{t:MSD_asymptotic_dist}: for the subdiffusive Hurst parameter value of $H=0.25$ $(\alpha=0.5)$, the distribution is distinctively Gaussian; by contrast, for the strongly superdiffusive value $H=0.9$ $(\alpha=1.8)$, the Rosenblatt-like attractor skews the finite-sample distribution. Moreover, Table \ref{table:all_lags_vs_two_lags} displays results for $\widehat{H} = \widehat{\alpha}_{\MSD}/2$ under a fBm. The simulations encompass two distinct situations. In the first one, we follow the common practice in microrheology of taking a large number of consecutive lag values such as $h_{\min},h_{\min}+1,\hdots,h_{\max}-1,h_{\max}$, where $h_{\min}=2$ and $h_{\max}=128$. In the second one, we pick only two lag values, namely $h_1 = 2$ and $h_2 = 128$. The results show that dropping most of the $\MSD$s has little impact on the performance of $\widehat{\boldsymbol \beta}_n$. Moreover, simulation studies not shown provide evidence that in both subdiffusive and superdiffusive ranges ($H=0.25,0.5,0.75,0.9$ or $\alpha = 0.5,1,1.5,1.8$, respectively), a pairwise combination of two low lag values (such as $h_1 = 1$ and $h_2 = 2$) leads to the best statistical performance, as measured by the Monte Carlo mean squared error.

As expected, though, the results for the ifOU are quite distinct. Since the latter is not exactly self-similar, the MSD curves display the asymptotic flattening effect revealed in Figure \ref{f:msd_plots} for $H = 0.25$ ($\alpha = 0.50$). This is what drives biophysicists to use larger lags when modeling anomalous diffusion data in the first place. Table \ref{table:small_lags_vs_large_lags} illustrates this effect in the estimation of $\widehat{{\boldsymbol \beta}}_n$ based on triples of consecutive lag values: the estimation bias decreases as the chosen lags increase. However, since the variance also increases with the lag, due to the smaller number of terms in $\overline{\mu}_2(h)$, the choice of regression lags with the lowest mean squared error turns out to be $[2^5,2^6,2^7]$. A different phenomenon emerges in the superdiffusive range. Table \ref{table:small_lags_vs_large_lags} shows that for very high values of $H$ or $\alpha$, it may be optimal to use lower regression lag values. This is so because the bias for large values of $H$ or $\alpha$ is very large. Though not displayed, the same issue appears under different parameter values in the superdiffusive range, and its cause is a matter for future investigation.

\begin{table}[!hbp]
\caption{Mean and s.d. of $\widehat{H}$ for fBm: consecutive lags ($2,\hdots,2^{8}$) vs two lags ($h_1 = 2$, $h_2 = 2^{7}$)}
\begin{center}
\begin{tabular}{ccccccccc}
 & $H = 0.25$ & & & & $H = 0.90$ & & & \\
  & $(\alpha=0.5)$ & & & & $(\alpha=1.8)$ & & & \\
\hline
$n$ &  $\widehat{E}_{\textnormal{cons}}(H)$ & $\widehat{\textnormal{sd}}_{\textnormal{cons}}(H)$  &  $\widehat{E}_2(H)$  & $\widehat{\textnormal{sd}}_2(H)$ & $\widehat{E}_{\textnormal{cons}}(H)$ & $\widehat{\textnormal{sd}}_{\textnormal{cons}}(H)$  &  $\widehat{E}_2(H)$  & $\widehat{\textnormal{sd}}_2(H)$ \\
\hline
$2^9$   &0.2363   & 0.0723 &   0.2376  &  0.0527 & 0.8380  &  0.1067   & 0.8459 &   0.0890 \\
$2^{10}$ &0.2437  &  0.0508 &  0.2446  &  0.0359  & 0.8583  &  0.0772  &  0.8644   & 0.0648 \\
$2^{11}$  &0.2468   & 0.0361 &  0.2476 &   0.0253  & 0.8715   & 0.0574 &  0.8751  &  0.0507  \\
$2^{12}$  &0.2482  &  0.0252&   0.2486  &  0.0177   & 0.8792 &   0.0462 &  0.8814 &   0.0415 \\
\hline
\end{tabular}\label{table:all_lags_vs_two_lags}
\end{center}
\end{table}

\begin{table}[!hbp]
\caption{Mean and s.d. of $\widehat{H}$ for ifOU: small lags vs large lags, path length=$2^{12}$}
\begin{center}
\begin{tabular}{ccccc}
 & $H = 0.25$ ($\alpha=0.5$) &  & $H = 0.90$ ($\alpha=1.8$) & \\
\hline
lags &  $\widehat{E}(H)$ & $\widehat{\textnormal{sd}}(H)$  &  $\widehat{E}(H)$ & $\widehat{\textnormal{sd}}(H)$\\
\hline
$[2^3,2^4,2^5]$   & 0.3218 & 0.0272  & 0.8817 & 0.0266  \\
$[2^4,2^5,2^6]$  & 0.2964 & 0.0354  & 0.8441 & 0.0392 \\
$[2^5,2^6,2^7]$   & 0.2756 & 0.0488  & 0.8315 & 0.0506 \\
$[2^6,2^7,2^8]$   & 0.2591 & 0.0742   & 0.8239 & 0.0704\\
\hline
\end{tabular}\label{table:small_lags_vs_large_lags}
\end{center}
\end{table}

\section{Conclusion}

In this paper, we establish the asymptotic distribution of the $\MSD$--based estimator widely used in the biophysical literature on anomalous diffusion. We assume that the particle undergoes a Gaussian, stationary-increment process, and take a pathwise approach, i.e., only one particle path is available. Depending on the diffusion exponent of the underlying process, the $\MSD$--based estimator has Gaussian or non-Gaussian limiting distribution, as well as different convergence rates. The asymptotic distributions provided allow for a new statistical perspective on the many numerical-experimental results reported by the biophysical community. We illustrate our results analytically and computationally based on fractional Brownian motion and the integrated fractional Ornstein-Uhlenbeck process.

\appendix

\section{Proofs for Section \ref{s:assumptions}}
We first show a lemma that will be used in the proof of Proposition \ref{prop:assumption_regu_shortmemo}.

\begin{lemma}\label{lem:short_memo_s(x)}
Suppose assumptions (A1) and (A2) hold. Then, for $k_1,k_2=1,...,m$ and $h_k$ as in \eqref{e:hk},
\begin{equation}\label{e:gamma_kl/h^alpha}
    \abs{\frac{\gamma_{h}(z)_{k_1,k_2}}{h^{\alpha}} - \frac{C_{\alpha}^2}{h^{\alpha}} \int_{\bbR} e^{ix z}
    \frac{(e^{ih_{k_1} x}-1)(e^{-ih_{k_2} x}-1)}{\abs{x}^{\alpha+1}} dx  }\leq C h^{-\delta},
\end{equation}
where $h,z\in \bbZ$, $h\geq\varepsilon_0^{-2}$, and $\delta=\min(\alpha/2,\delta_0/2)>0$ (see \eqref{x_spec_rep} and \eqref{e:s(x)}).
\end{lemma}

\begin{proof}
By \eqref{x_spec_rep}, we obtain a harmonizable representation for the size $h$ increment process $Y_z(h)$, namely,
\begin{equation}\label{e:increment_process_harmonizable}
Y_z(h) = X(z + h) - X(z) = C_\alpha \int_{\bbR} e^{izx} \frac{e^{ihx}-1}{ix} \frac{s(x)}{ |x|^{\alpha/2 - 1/2}}\widetilde{B}(dx), \quad z \in \bbZ.
\end{equation}
Fix $k_1$ and $k_2$. For notational simplicity, we will use the indices $k_1 = 1$, $k_2 = 2$. From \eqref{e:increment_process_harmonizable}, after the change of variables $xh=y$, we can write the covariance between the increments $Y_z(h_{k_1})$ and $Y_z(h_{k_2})$ as
$$
\gamma_{h}(z)_{1,2} = h^{\alpha} C_{\alpha}^2 \int_{\bbR} e^{i y z/h} \frac{(e^{iw_{1} y}-1)(e^{-iw_{2} y}-1)}{\abs{y}^{\alpha+1}} \abs{s(y/h)}^2 dy.
$$
Now break up the left-hand side of the expression \eqref{e:gamma_kl/h^alpha} into the sum
$$
\abs{ C_{\alpha}^2 \Big\{\int_{\abs{y}\leq\sqrt{h}} + \int_{\abs{y}>\sqrt{h}} \Big\} e^{iy z/h}
    \frac{(e^{iw_{1} y}-1)(e^{-iw_{2} y}-1)}{\abs{y}^{\alpha+1}} (\abs{s(y/h)}^2-1) dy  }=\abs{I_1+I_2},
$$
where $I_1$ and $I_2$ denote the integrals over the domains $(-\sqrt{h},\sqrt{h})$ and $\bbR \backslash (-\sqrt{h},\sqrt{h})$, respectively. Then, for $h\geq\varepsilon_0^{-2}$,
$$
\abs{I_1} \leq C \int_{\abs{y}\leq\sqrt{h}}  \frac{\abs{e^{iw_{1} y}-1}\abs{e^{-iw_{2} y}-1}}{\abs{y}^{\alpha+1}} \hspace{0.5mm}|\abs{s(y/h)}^2-1|\hspace{0.5mm} dy
$$
$$
\leq C \int_{\abs{y}\leq\sqrt{h}} \frac{\abs{e^{iw_{1} y}-1}\abs{e^{-iw_{2} y}-1}}{\abs{y}^{\alpha+1}} \abs{y/h}^{\delta_0} dy
$$
\begin{equation}\label{e:I1_bound}
\leq C h^{-\delta_0/2} \int_{\bbR}  \frac{\abs{e^{iw_{1} y}-1}\abs{e^{-iw_{2} y}-1}}{\abs{y}^{\alpha+1}} dy=C h^{-\delta_0/2}.
\end{equation}
Moreover, since $s(x)$ is bounded,
\begin{equation}\label{e:I2_bound}
\abs{I_2}
\leq C \int_{\abs{y}>\sqrt{h}}  \frac{4M}{\abs{y}^{\alpha+1}}dy
\leq C h^{-\alpha/2}.
\end{equation}
The expressions \eqref{e:I1_bound} and \eqref{e:I2_bound} yield \eqref{e:gamma_kl/h^alpha}. $\Box$\\
\end{proof}

\noindent {\sc Proof of Proposition \ref{prop:assumption_regu_shortmemo}}: Fix $k_1$ and $k_2$. For notational simplicity, we will use the indices $k_1 = 1$ and $k_2 = 2$. To show ($i$), set $z=0$ and $h_{k_1}=h_{k_2}=h$ in \eqref{e:gamma_kl/h^alpha}. Then,
$$
 \abs{\frac{E X^2(h) }{\theta h^{\alpha}}-1}= \abs{\frac{\gamma_{h}(0)_{1,1}}{\theta h^{\alpha}}-1} \leq C h^{-\delta},\quad h\rightarrow \infty,
$$
where $\theta=C_{\alpha}^2  \int_{\bbR} {|e^{iy}-1|^2}{\abs{y}^{-(\alpha+1)}}dy $.

We now show ($ii$). The proof draws upon conveniently rewriting the integral term in \eqref{e:gamma_kl/h^alpha} based on the closed form expression for the covariance of a (standard) fBm.
In fact, recall that $\tau = (\frac{C_{\alpha}}{C_{H}})^2 \frac{\alpha(\alpha-1)}{2}$, by \eqref{e:tau}. Then,
$$
\Big| \frac{\gamma_h(z)_{1,2}}{h^{\alpha}} - w_{1}w_{2} \tau \Big( \frac{z}{h}\Big)^{\alpha-2}\Big|
$$
$$
\leq \Big| \frac{\gamma_h(z)_{1,2}}{h^{\alpha}} - \frac{C_{\alpha}^2}{h^{\alpha}} \int_{\bbR} e^{ixz} \frac{(e^{ih_{1} x}-1)(e^{-ih_{2} x}-1)}{\abs{x}^{\alpha+1}}dx \Big|
$$
\begin{equation}\label{e:cov_decay_proof_main_term}
 + \Big|\Big(\frac{C_{\alpha}}{C_{H}}\Big)^2 \Big\{\frac{C_{H}^2}{h^{\alpha}} \int_{\bbR} e^{ixz} \frac{(e^{ih_{1} x}-1)(e^{-ih_{2} x}-1)}{\abs{x}^{\alpha+1}} dx \Big\} -
w_1w_2 \tau \Big( \frac{z}{h}\Big)^{\alpha-2}\Big|.
\end{equation}
Note that
\begin{equation}\label{e:covariance_size_h_increment_process_standard_fBm}
\frac{C_{H}^2}{h^{\alpha}}\int_{\bbR} e^{ixz} \frac{(e^{ih_{1} x}-1)(e^{-ih_{2} x}-1)}{\abs{x}^{\alpha+1}} dx
\end{equation}
is the expression for the covariance $\gamma_h(z)_{1,2}$ of a size $h$ increment process $Y_z(h)$ (see \eqref{e:increment_process_harmonizable}) formed from a standard fBm $B_H$. Pick $|z| \geq h_{m}+1$. If $\alpha = 1$, the integral \eqref{e:covariance_size_h_increment_process_standard_fBm} is identically zero, by the independence of non-overlapping increments. Alternatively, when $\alpha\neq 1$,
the closed form \eqref{e:fBm_cov} with $\sigma^2 = 1$ allows us to rewrite \eqref{e:covariance_size_h_increment_process_standard_fBm} as
$$
\frac{1}{h^{\alpha}}[EB_H(z + h_{1})B_H(h_{2}) - EB_H(z)B_H(h_{2})] = \frac{1}{2h^{\alpha}}\{|z + h_{1}|^{\alpha} - |z + h_{1} - h_{2}|^{\alpha}- |z|^{\alpha} + |z - h_{2}|^{\alpha}\}
$$
\begin{equation}\label{e:e:covariance_size_h_increment_process_standard_fBm_modified}
= \frac{|z|^{\alpha}}{2h^{\alpha}} \Big\{ \Big(\Big| 1 + w_{1}\frac{h}{z}\Big|^{\alpha} - 1 \Big) - \Big( \Big| 1 + (w_{1} - w_{2})\frac{h}{z}\Big|^{\alpha} - 1 \Big) + \Big( \Big| 1 - w_{2}\frac{h}{z}\Big|^{\alpha} - 1 \Big) \Big\}.
\end{equation}
Let $f(x) = x^{\alpha}$. Based on second order Taylor expansions of $f$ around 1, we can recast the expression \eqref{e:e:covariance_size_h_increment_process_standard_fBm_modified} as
$$
\frac{|z|^{\alpha}}{2h^{\alpha}} \Big\{\frac{\alpha(\alpha-1)}{2}\Big( \frac{h}{z}\Big)^2[w^2_{1} - (w_{1} - w_{2})^2 + w^2_{2}] + O\Big[ \Big( \frac{h}{z}\Big)^3\Big]\Big\}
$$
\begin{equation}\label{e:decay_cov_bound_2nd_term}
= \Big(\frac{|z|}{h}\Big)^{\alpha-2} \frac{\alpha(\alpha-1) w_{1} w_{2}}{2}\Big( 1 + O\Big( \frac{h}{z}\Big)\Big).
\end{equation}
Therefore, based on \eqref{e:decay_cov_bound_2nd_term} (which also encompasses the case $\alpha = 1$) and \eqref{e:gamma_kl/h^alpha}, the expression \eqref{e:cov_decay_proof_main_term} can be further bounded by $C [ (\frac{|z|}{h})^{\alpha-3}+ h^{-\delta} ]$. As a consequence, we arrive at
$$
\abs{ \frac{\gamma_{h}(z)_{1,2}}{|z|^{\alpha-2}h_1 h_2} - \tau} \leq C \bigg( \frac{h}{|z|}+ \bigg(\frac{h}{|z|}\bigg)^{\alpha} \frac{z^{2}}{h^{2+\delta}} \bigg)
\leq C \bigg(\frac{h}{|z|}\bigg)^{\min(1,\alpha)} \leq C \bigg(\frac{h}{|z|}\bigg)^{\delta},
$$
where the last two inequalities result from \eqref{e:h(n)} and \eqref{e:hm+1<=z<=n}. Setting $g(z,h)=\frac{\gamma_{h}(z)_{1,2}}{|z|^{\alpha-2}h_{1} h_{2}} - \tau$ yields \eqref{e:gamma=z2H-2_hk_hl_const+resid}. $\Box\\$

\section{Proofs for Section \ref{s:MSD_distribution}}

{\sc Proof of Theorem \ref{t:MSD_asymptotic_dist}}:
In view of \eqref{e:h(n)}, the claim is equivalent to
$$
    \bigg( \eta^{-1}(n)\zeta^{-1}(h_k)\sum_{i_k=1}^{n}
        \{Y_{i_k}^2(h_k)-EX^2(h_k)\}\bigg)_{k=1,\hdots,m}
        \stackrel{d}{\rightarrow} {\mathbf{Z}}.
$$
Consider the vector of increments
$$
{\mathbf Y}=(Y_1(h_1),\hdots,Y_{n}(h_1);Y_1(h_2),\hdots,Y_{n}(h_2);\hdots;Y_1(h_m),\hdots,Y_{n}(h_m))^T,
$$
The covariance matrix of $\mathbf{Y}$ can be written as $R_m(n)=(R_{k_1,k_2}(n))_{k_1,k_2=1,...,m}$, where
$$
  R_{k_1,k_2}(n) :=\left(
              \begin{array}{ccc}
                \gamma_{h}(0)_{k_1,k_2} & \cdots & \gamma_{h}(1-n)_{k_1,k_2} \\
                \vdots & \ddots & \vdots \\
                \gamma_{h}(n-1)_{k_1,k_2} & \cdots & \gamma_{h}(0)_{k_1,k_2}  \\
              \end{array}
            \right) \in \bbR^{n^2}.
$$
Let
\begin{equation}\label{e:z_n}
{\mathbf Z}_n=(Z_n(h_1),\hdots,Z_n(h_m))^T
\end{equation}
be the centered statistic defined by
\begin{equation*}
    Z_n(h_k)=\eta^{-1}(n)\zeta^{-1}(h_k)\sum_{i_k=1}^{n}\{Y_{i_k}^2(h_k)-EX^2(h_k)\}, \quad k = 1, \hdots, m.
\end{equation*}
Also, let
$$
    D_{m}(n):=\textnormal{diag}\Big(D_{1,1}(n),D_{2,2}(n),...,D_{m,m}(n)\Big), \quad
    D_{i,i}(n):=\frac{t_i}{\zeta(h_i)}I_n, \quad i=1,...,m,
$$
where $I_n$ denotes an $n\times n$ identity matrix. The weak limit \eqref{e:MSD_asymptotic_dist} can be established via characteristic functions. The initial manipulation of the characteristic function is very similar to that in Rosenblatt \cite{rosenblatt:1961}. First note that
$$
\int_{\bbR^n} \exp\Big\{ - \frac{1}{2}\textbf{y}^T(R^{-1}_{m}(n) - i c D_{m}(n))\textbf{y}\Big\}d\textbf{y} = (2 \pi)^{n/2} \det(R_m^{-1}(n)-icD_m(n))^{-1/2}, \quad c \in \bbR.
$$
By a similar computation to that in Taqqu \cite{taqqu:2011}, pp.42--43,
$$
\phi_{\textbf{Z}_n}(\textbf{t}) =  E( e^{i{\mathbf t}^T {\mathbf Z}_n} ) = E e^{i {\mathbf t}^T (Z_{n}(h_1),Z_{n}(h_2),\hdots, Z_{n}(h_m))^T}
$$
$$
= \int_{\bbR^{mn}}\exp\Big\{i\sum_{k=1}^{m}t_k z_n(h_k)\Big\}
        \frac{1}{\sqrt{\det (2\pi R_{m}(n))}} \exp\Big\{-\frac{1}{2}{\mathbf y}^T R_{m}^{-1}(n){\mathbf y}\Big\} d{\mathbf y}
$$
\begin{equation}\label{e:chf_in_terms_of_log}
= \exp\Big\{ \frac{1}{2} \Big[-2i \eta^{-1}(n)\sum_{k=1}^m nt_k \zeta^{-1}(h_k) \gamma_h(0)_{k,k} -\sum_{l=1}^{mn}\log(1-2i \eta^{-1}(n)\lambda_{l,mn})\Big] \Big\}.
\end{equation}
The scalars $\lambda_{l,mn}$, $l=1,\hdots,mn$, denote the eigenvalues (characteristic roots) of $R_{m}(n)D_{m}(n) = PJP^{-1}$, where $P\in GL(mn,\mathbb{C})$, and $J$ is in Jordan form. By the analytic expansion of $  \log(1-2i \eta^{-1}(n)\lambda_{l,mn})$,
\begin{equation}\label{e:log_expansion}
    -\sum_{l=1}^{mn}\log(1-2i \eta^{-1}(n)\lambda_{l,mn})=2i \eta^{-1}(n)\sum_{l=1}^{mn}\lambda_{l,mn}
        + \sum_{s=2}^{\infty} \frac{(2i)^s}{s} \eta^{-s}(n)\sum_{l=1}^{mn}\lambda_{l,mn}^s.
\end{equation}
However, $\sum_{l=1}^{mn}\lambda_{l,mn}=\textnormal{tr}(R_{m}(n)D_{m}(n))=\sum_{k=1}^m n t_k \zeta^{-1}(h_k) \gamma_h(0)_{k,k}$. Thus, by \eqref{e:log_expansion} we can rewrite \eqref{e:chf_in_terms_of_log} as
\begin{equation}\label{e:char_z_n}
    E e^{i{\mathbf t}^T{\mathbf Z}_n}
        = \exp\Big\{ \frac12 \sum_{s=2}^{\infty} \frac{(2i)^s}{s} \eta^{-s}(n)\sum_{l=1}^{mn}\lambda_{l,mn}^s \Big\}.
\end{equation}
Moreover,
$$
\eta^{-s}(n)\sum^{mn}_{l=1} \lambda^{s}_{l,mn} = \eta^{-s}(n)\textnormal{tr}[(R_{m}(n)D_{m}(n))^s]
$$
$$
=\eta^{-s}(n)\sum_{k_1,...,k_s=1}^{m}\Big\{
    \tr\Big[ R_{k_1,k_2}(n)D_{k_2,k_2}(n)R_{k_2,k_3}(n)D_{k_3,k_3}(n)\hdots R_{k_s,k_1}(n)D_{k_1,k_1}(n) \Big]\Big\}
$$
\begin{equation}\label{eqn:trace}
    \begin{split}
    = & \sum_{k_1,...,k_s=1}^{m}\Big\{ t_{k_1}t_{k_2} \cdots t_{k_s}
        \zeta^{-1}(h_{k_1})\zeta^{-1}(h_{k_2}) \cdots \zeta^{-1}(h_{k_s})\\
    & \times \eta^{-s}(n)\sum_{i_1,...,i_s=1}^{n} \gamma_h(i_1-i_2)_{k_1,k_2}
     \gamma_h(i_2-i_3)_{k_2,k_3}\cdots\gamma_h(i_s-i_1)_{k_s,k_1} \Big\}.
    \end{split}
\end{equation}
The weak limits \eqref{e:Sigma_0<H<3/4}, \eqref{e:Sigma_H=3/4} and \eqref{e:3/4<H<1_multivariate_Rosen_limit} are a consequence of Propositions \ref{prop:3/4<H<1} and \ref{prop:0<H<3/4}. $\Box$\\

\noindent {\sc Proof of Corollary \ref{t:MSD_asymptotic_coro}} We first show that
    \begin{equation}\label{e:MSD_asymptotic_various_lags}
    \frac{n h^{\alpha}}{\eta(n)\zeta(h)}
    \bigg( \frac{\overline{\mu}_2(h_k)}{\theta h_k^{\alpha}}-1 \bigg)_{k=1,\hdots,m}
    \stackrel{d}{\rightarrow} A{\mathbf{Z}},
    \end{equation}
    where $A=A(\theta,\alpha)$ and ${\mathbf{Z}}$ are as in \eqref{e:A(theta,alpha)} and Theorem \ref{t:MSD_asymptotic_dist}, respectively. Based on \eqref{e:mu2(hk)}, rewrite the left-hand side of the expression (\ref{e:MSD_asymptotic_dist}) as
$$
\bigg( \frac{N_k \theta h_k^{\alpha}}{\eta(N_k)\zeta(h_k)}\Big( \frac{\overline{\mu}_2(h_k)}{\theta h_k^{\alpha}} - \frac{EX^2(h_k)}{\theta h_k^{\alpha}}
 \Big)\bigg)_{k=1,\hdots,m}.
$$
This random vector has the same asymptotic distribution as
\begin{equation}\label{e:MSD_asymptotic_plus_bias}
\frac{ n h^{\alpha}}{\eta(n)\zeta(h)}
\bigg( \frac{ \theta w_k^{\alpha} }{ \zeta(w_k)}
\Big( \frac{\overline{\mu}_2(h_k)}{\theta h_k^{\alpha}} - 1\Big)\bigg)_{k=1,\hdots,m}
-\frac{ n h^{\alpha}}{\eta(n)\zeta(h)}
\bigg( \frac{ \theta w_k^{\alpha} }{ \zeta(w_k)}
\Big( \frac{EX^2(h_k)}{\theta h_k^{\alpha}} - 1\Big) \bigg)_{k=1,\hdots,m}.
\end{equation}
However, the bound \eqref{e:gamma=<h^2H_0} yields
\begin{equation}\label{e:MSD_asymptotic_bias_go_0}
\frac{ n h^{\alpha}}{\eta(n)\zeta(h)} \frac{ \theta w_k^{\alpha} }{ \zeta(w_k)}
\Big| \frac{EX^2(h_k)}{\theta h^{\alpha}_k} - 1\Big| \leq \frac{ n h^{\alpha}}{\eta(n)\zeta(h)} h^{-\delta} \rightarrow 0, \quad k=1,\hdots,m,
\end{equation}
where the zero limit is a consequence of \eqref{e:h(n)} and \eqref{e:eta_zeta}. The expression \eqref{e:MSD_asymptotic_various_lags} is now a consequence of \eqref{e:MSD_asymptotic_plus_bias}, \eqref{e:MSD_asymptotic_bias_go_0} and \eqref{e:MSD_asymptotic_dist}.\\

To show \eqref{e:lm_estimator_asymptotic}, rewrite
\begin{equation}\label{e:hatbeta-beta}
\widehat{\boldsymbol \beta}_n-{\boldsymbol \beta}=(M_n^T M_n)^{-1}M_n^T (Q_n-M_n{\boldsymbol \beta}).
\end{equation}
By entrywise first order Taylor expansions,
\begin{equation}\label{e:Qn-Mnbeta}
Q_n-M_n {\boldsymbol \beta}=\bigg( \log\Big(\frac{\overline{\mu}_2(h_k)}{\theta h_k^{\alpha}}\Big) \bigg)_{k=1,\hdots,m}
= \bigg( \frac{\overline{\mu}_2(h_k)}{\theta h_k^{\alpha}}-1 \bigg)_{k=1,\hdots,m}
+O\bigg( \frac{\overline{\mu}_2(h_k)}{\theta h_k^{\alpha}}-1 \bigg)^2_{k=1,\hdots,m}
\end{equation}
On the other hand, note that $\det(M_n^T M_n)= c_w$ (see \eqref{e:c_det_of_MTM}) is a constant with respect to $n$. Thus,
$$
(M_n^T M_n)^{-1}M_n^T= \frac{1}{c_w}
\left(
  \begin{array}{cc}
    \sum_{k=1}^{m} \log^2 h_k  & -\sum_{k=1}^{m} \log h_k \\
    -\sum_{k=1}^{m} \log h_k & m \\
  \end{array}
\right)
\left(
  \begin{array}{ccc}
    1 & \hdots & 1 \\
    \log h_1 & \hdots & \log h_m \\
  \end{array}
\right)
$$
$$
=\frac{1}{c_w}
\left(
  \begin{array}{ccc}
    \sum_{k=1}^{m} \log^2 h_k -\log h_1 \sum_{k=1}^{m} \log h_k  & \hdots & \sum_{k=1}^{m} \log^2 h_k -\log h_m \sum_{k=1}^{m} \log h_k  \\
    m\log h_1 -\sum_{k=1}^{m} \log h_k  &\hdots & m\log h_m -\sum_{k=1}^{m} \log h_k  \\
  \end{array}
\right).
$$
Moreover, for $j=1,...,m$,
$$
\sum_{k=1}^{m} \log^2 h_k -\log h_j \sum_{k=1}^{m} \log h_k =\log h \sum_{k=1}^{m} \log(w_k/w_j)+\sum_{k=1}^{m} \log w_k \log(w_k/w_j)
$$
and $m\log h_j -\sum_{k=1}^{m} \log h_k =\sum_{k=1}^{m} \log(w_j/w_k)$. Therefore, by \eqref{e:h(n)}, we obtain the entrywise asymptotic equivalence
\begin{equation}\label{e:MnTMn-1Mn}
(M_n^T M_n)^{-1}M_n^T \sim \frac{1}{c_w} \left(
                                \begin{array}{cc}
                                  \log h & 0 \\
                                  0 & 1 \\
                                \end{array}
                              \right)
\left(
  \begin{array}{ccc}
    \sum_{k=1}^{m} \log(w_k/w_1) & ... & \sum_{k=1}^{m} \log(w_k/w_m) \\
    \sum_{k=1}^{m} \log(w_1/w_k) & ... & \sum_{k=1}^{m} \log(w_m/w_k) \\
  \end{array}
\right).
\end{equation}
By \eqref{e:hatbeta-beta}, \eqref{e:Qn-Mnbeta}, \eqref{e:MnTMn-1Mn}, \eqref{e:MSD_asymptotic_various_lags}, and \eqref{e:u_and_v}, we arrive at \eqref{e:lm_estimator_asymptotic}.
$\Box$\\

\section{Auxiliary results}\label{s:auxiliary}

Lemmas \ref{lem:3/4<H<1_neardiag_h}-\ref{lem:0<H<3/4_s=2}, stated below, are used in the proofs in Propositions \ref{prop:3/4<H<1} and \ref{prop:0<H<3/4}. The proofs of the lemmas can be found in Section \ref{s:auxiliary_proofs}.

\begin{lemma}\label{lem:3/4<H<1_neardiag_h}
Consider $3/2<\alpha<2$ and $s \geq 2$, and suppose the assumptions (A1) and (A2) hold. Then, as $n\rightarrow \infty$,
\begin{equation}\label{e:0<H<3/4_sum_gamma}
    \zeta^{-1}(h_{k_1})\cdots \zeta^{-1}(h_{k_s}) \eta^{-s}(n)
    \sum_{\underset{\abs{i_1-i_2}\leq h\cup...\cup\abs{i_s-i_1}\leq h}{i_1,...,i_s=1}}^{n}
    \gamma_h(i_1-i_2)_{k_1,k_2}\cdots\gamma_h(i_s-i_1)_{k_s,k_1}
    \rightarrow 0.
\end{equation}
\end{lemma}

\begin{lemma}\label{lem:3/4<H<1_offdiag_h}
Consider $3/2<\alpha <2$ and $s\geq 2$, and suppose the assumptions (A1) and (A2) hold. Then, as $n\rightarrow \infty$,
$$
    \zeta^{-1}(h_{k_1})\cdots \zeta^{-1}(h_{k_s}) \eta^{-s}(n)
    \sum_{\underset{\abs{i_1-i_2}\geq h+1,...,\abs{i_s-i_1}\geq h+1 }{i_1,...,i_s=1}}^{n}
    \gamma_h(i_1-i_2)_{k_1,k_2}\cdots\gamma_h(i_s-i_1)_{k_s,k_1}
$$
\begin{equation}\label{e:mult_sum_gamma_goes_to_mult_integ_3/4<H<1}
\rightarrow \tau^s \int^{1}_0 \hdots \int^{1}_0 |x_1 - x_2 |^{\alpha-2}\hdots|x_s - x_1 |^{\alpha-2}dx_1 \hdots dx_s.
\end{equation}
\end{lemma}

\begin{lemma}\label{lem:0<H<3/4_neardiag_h}
Consider $0 < \alpha \leq 3/2$ and $s \geq 3$, and suppose the assumptions (A1) and (A2) hold. Then, as $n\rightarrow \infty$,
\begin{equation}\label{e:0<H<3/4_neardiag_h_original_sum}
    \zeta^{-1}(h_{k_1})\cdots \zeta^{-1}(h_{k_s}) \eta^{-s}(n)
    \sum_{i_1,...,i_s=1}^{n}
    \gamma_h(i_1-i_2)_{k_1,k_2}\cdots\gamma_h(i_s-i_1)_{k_s,k_1}
    \rightarrow 0.
\end{equation}
\end{lemma}

\begin{lemma}\label{lem:0<H<3/4_s=2}
Suppose the assumptions (A1) and (A2) hold. Then, as $n\rightarrow \infty$,
\begin{itemize}
  \item [(i)] in the parameter range $0<\alpha<3/2$,
$$
    \eta^{-2}(n)\zeta^{-1}(h_{k_1})\zeta^{-1}(h_{k_2}) \sum_{i_1,i_2=1}^{n} \gamma^2_h(i_1-i_2)_{k_1,k_2}
$$
\begin{equation}\label{e:0<H<3/4_s=2}
    \rightarrow w_{k_1}^{-(\alpha+1/2)}w_{k_2}^{-(\alpha+1/2)}\Big(\frac{C_\alpha}{C_H}\Big)^4 \norm{\widehat{G}(y;w_{k_1},w_{k_2})}^2_{L^2(\bbR)},
\end{equation}
where $\widehat{G}(y;w_{k_1},w_{k_2})$, $C_\alpha$ and $C_H$ are defined by \eqref{e:G_hat}, \eqref{x_spec_rep} and \eqref{e:CH_standard_fBm}, respectively;
  \item [(ii)] when $\alpha=3/2$,
\begin{equation}\label{e:H=3/4_s=2}
    \eta^{-2}(n)\zeta^{-1}(h_{k_1})\zeta^{-1}(h_{k_2}) \sum_{i_1,i_2=1}^{n} \gamma^2_h(i_1-i_2)_{k_1,k_2} \rightarrow 2\tau^2,
\end{equation}
where $\tau$ is given by \eqref{e:tau}.
\end{itemize}
\end{lemma}


\begin{proposition}\label{prop:3/4<H<1}
Consider the parameter range $3/2 < \alpha < 2$ and suppose the assumptions (A1)--(A2) hold. Then, as $n\rightarrow \infty$, the vector ${\mathbf{Z}_n}=(Z_n(h_1),Z_n(h_2),...Z_n(h_m))^T$ in \eqref{e:z_n} converges in law to a Rosenblatt-like distribution whose characteristic function is given by
\begin{equation}\label{e:3/4<H<1}
\phi_{{\mathbf{Z}}}(\mathbf{t}) = \exp\Big\{\frac{1}{2} \sum^{\infty}_{s=2} \frac{(2i\tau \sum^{m}_{k=1} t_k)^s}{s} \int^{1}_0 \hdots \int^{1}_0 |x_1 - x_2 |^{\alpha-2}\hdots|x_s - x_1 |^{\alpha-2}dx_1 \hdots dx_s  \Big\}.
\end{equation}
\end{proposition}
\begin{proof}
Let $s \geq 2$ and consider the expression \eqref{eqn:trace}. By Lemmas \ref{lem:3/4<H<1_neardiag_h} and \ref{lem:3/4<H<1_offdiag_h}, as $n\rightarrow \infty$ the right-hand side of the latter converges to
$$
\sum_{k_1,...,k_s=1}^{m} t_{k_1}t_{k_2} \cdots t_{k_s}\tau^s \int^{1}_0 \hdots \int^{1}_0|x_1 - x_2 |^{\alpha-2}\hdots|x_s - x_1 |^{\alpha-2}dx_1 \hdots dx_s
$$
$$
= \Big( \sum^{m}_{k=1}t_k \Big)^s\tau^s \int^{1}_0 \hdots \int^{1}_0 |x_1 - x_2 |^{\alpha-2}\hdots|x_s - x_1 |^{\alpha-2}dx_1 \hdots dx_s.
$$
Therefore, the characteristic function \eqref{e:char_z_n} converges to \eqref{e:3/4<H<1}, as claimed.
$\Box$\\
\end{proof}

\begin{proposition}\label{prop:0<H<3/4}
For $0 < \alpha \leq 3/2$, suppose the assumptions (A1)--(A2) hold. Let ${\mathbf{Z}_n}=(Z_n(h_1),Z_n(h_2),...Z_n(h_m))^T$ be the random vector in \eqref{e:z_n}. Then, as $n\rightarrow \infty$, ${\mathbf{Z}_n}\overset{d}{\rightarrow}N(\mathbf{0},\Sigma)$, where $\Sigma$ is a $m\times m$ matrix with components
\begin{equation}\label{e:Sigma_k1,k2}
\Sigma_{k_1,k_2}=\left\{
                   \begin{array}{ll}
                   2w_{k_1}^{-\alpha-1/2}w_{k_2}^{-\alpha-1/2} (\frac{C_\alpha}{C_H})^4 \norm{\widehat{G}(y;w_{k_1},w_{k_2})}^2_{L^2(\bbR)}, & 0 < \alpha < 3/2; \\
                    4 \tau^2, & \alpha=3/2,
                   \end{array}
                 \right.
\end{equation}
and $\widehat{G}(y;w_{k_1},w_{k_2})$ is defined by \eqref{e:G_hat}.
\end{proposition}
\begin{proof}
When $0 < \alpha \leq 3/2$, by Lemma \ref{lem:0<H<3/4_neardiag_h} it suffices to consider the term \eqref{eqn:trace} of order $s = 2$. Therefore, by Lemma \ref{lem:0<H<3/4_s=2}, as $n \rightarrow \infty$ the characteristic function \eqref{e:char_z_n} converges to that of a multivariate normal distribution with covariance matrix $\Sigma = (\Sigma_{k_1,k_2})_{k_1,k_2=1,\hdots,m}$ as in \eqref{e:Sigma_k1,k2}. $\Box$\\
\end{proof}

\section{Additional proofs}\label{s:auxiliary_proofs}

This section contains the proofs of Lemmas \ref{lem:3/4<H<1_neardiag_h}--\ref{lem:0<H<3/4_s=2}.\\

For $h_m+1 \leq |z| \leq n$, recall that conditions \eqref{e:gamma=z2H-2_hk_hl_const+resid} and \eqref{e:residual_function_g} can be jointly expressed as
\begin{equation}\label{e:hk_neq_hl}
\gamma_{h}(z)_{k_1,k_2}=w_{k_1} w_{k_2} \abs{z}^{\alpha-2}h^2  \{\tau + g(z,h)_{k_1,k_2} \}, \quad
\abs{g(z,h)_{k_1,k_2}} \leq C \bigg(\frac{h}{|z|}\bigg)^{\delta},
\end{equation}
for a general pair of indices $k_1,k_2 = 1,\hdots,m$ representing shifting lag values. Moreover, by the Cauchy-Schwarz inequality and \eqref{e:gamma=<h^2H_0},
\begin{equation}\label{e:hk_neq_hl_assumption_gamma_h_z<h}
\abs{\gamma_{h}(z)_{k_1,k_2}} \leq C h^{\alpha}, \quad h, z \in \bbZ,
\end{equation}
where $C > 0$ does not depend on $k_1$, $k_2$. In particular, for a single shifting lag value
\begin{equation}\label{e:hk=hl=h(n)}
h_{k_1} = h_{k_2} = h(n)=:h,
\end{equation}
the expressions \eqref{e:hk_neq_hl} and \eqref{e:gamma=<h^2H_0} imply that
\begin{equation}\label{e:gamma_h(z)}
\gamma_h(z) := \gamma_{h}(z)_{1,1}= \abs{z}^{\alpha-2}h^2 \{\tau+g(z,h) \}, \quad \abs{g({z},h )}\leq C \Big(\frac{h}{|z|}\Big)^{\delta}.
\end{equation}
Thus, in the proofs of Lemmas \ref{lem:3/4<H<1_neardiag_h}--\ref{lem:0<H<3/4_s=2} below, we will first establish the statements for a single index (shifting lag value) $m = 1$ and \eqref{e:hk=hl=h(n)}, and then adjust the constants to obtain the general statements for $m > 1$. In the generalization it will always be implicit that where a multiple summation is taken over index ranges of the form $|i_1 - i_2| \geq h + 1$ or $|i_1 - i_2| \leq h $ under $m = 1$, one should substitute $h_m$ for $h$ under $m > 1$.\\

\noindent{\sc Proof of Lemma \ref{lem:3/4<H<1_neardiag_h}}
First assume $m = 1$. We only look at the subcase where the summation is taken over the index set
\begin{equation}\label{e:summ_range_|i1-i2|=<h}
\{|i_1 - i_2| \leq h\} \cap  \{|i_2 - i_3|\geq h+1\} \cap \hdots  \cap \{|i_s - i_1|\geq h+1\},
\end{equation}
since the remaining $2^s - 2$ subcases can be tackled in a similar fashion. By \eqref{e:gamma_h(z)}, we can rewrite the expression of interest as
$$
 \frac{h^{2(s-1)}}{\eta^s(n) \zeta^s(h)} \sum^{n}_{\stackrel{i_1,i_2,\hdots,i_s=1}{|i_1 - i_2|\leq h ,\hspace{0.5mm}|i_2 - i_3| \geq h+1, \hspace{0.5mm}\hdots \hspace{0.5mm}, |i_s - i_1|\geq h+1}} \gamma_h(i_1-i_2)
$$
\begin{equation}\label{e:3/4<H<1_exp_1}
\abs{i_2 - i_3}^{\alpha-2}\hdots \abs{i_s - i_1}^{\alpha-2}
\{ \tau + g( i_2-i_3,h ) \} \hdots \{ \tau + g( i_s-i_1,h ) \},
\end{equation}
where, under the summation sign, the terms of the form $\tau + g(\cdot,\cdot)$ can be uniformly bounded by a constant, and $\gamma_h(i_1-i_2)$ is bounded by $C h^{\alpha}$ (see \eqref{e:hk_neq_hl_assumption_gamma_h_z<h}). Thus, the absolute value of \eqref{e:3/4<H<1_exp_1} is bounded by
$$
C n^{s(1-\alpha)} h^{-2} \hspace{1mm} n^{(s-1)(\alpha-1)}\sum^{n}_{\stackrel{i_1,\hdots,i_s=1, i_1\neq i_2}{|i_1 - i_2|\leq h,|i_2 - i_3|\geq h+1, \hdots, |i_s - i_1|\geq h+1 }} h^{\alpha} \abs{\frac{i_2 - i_3}{n}}^{\alpha-2}\hdots \abs{\frac{i_s - i_1}{n}}^{\alpha-2} \frac{1}{n^{s-1}}
$$
$$
= C \frac{h^{\alpha-2}}{n^{\alpha-1} } \hspace{0.5mm}\sum^{h}_{z= - h}\sum^{n}_{\stackrel{i_2,\hdots,i_s=1}{|i_2 - i_3|\geq h+1, \hdots, |i_s - i_2 + z|\geq h+1 }} \abs{\frac{i_2 - i_3}{n}}^{\alpha-2}\hdots \abs{\frac{i_s - i_2 + z}{n}}^{\alpha-2} \frac{1}{n^{s-1}}
$$
$$
\leq C \Big(\frac{h}{n}\Big)^{\alpha-1} \sum^{n}_{\stackrel{i_2,\hdots,i_s=1}{|i_2 - i_3|\geq h+1, \hdots, |i_s - i_2 + z|\geq h+1 }} \abs{\frac{i_2 - i_3}{n}}^{\alpha-2}\hdots \abs{\frac{i_s - i_2 - \textnormal{sign}(i_s - i_2) h}{n}}^{\alpha-2} \frac{1}{n^{s-1}}
$$
$$
\sim C \Big(\frac{h}{n}\Big)^{\alpha-1}  \int^{1}_0 \hdots \int^{1}_0 |x_2 - x_3|^{\alpha -2 }\hdots |x_s - x_2|^{\alpha -2 } dx_2 \hdots dx_{s},
$$
which goes to zero as $n \rightarrow \infty$, since $\alpha > 3/2$ and by \eqref{e:h(n)}. This shows \eqref{e:0<H<3/4_sum_gamma} for $m = 1$. In addition, adjusting for the constants $w_{k_1}$, $w_{k_2}$ from \eqref{e:hk_neq_hl} does not alter the zero limit. Hence, \eqref{e:0<H<3/4_sum_gamma} also holds for $m > 1$. $\Box$\\


\noindent{\sc Proof of Lemma \ref{lem:3/4<H<1_offdiag_h}}
First assume $m = 1$. We start out by establishing that
$$
\sum^{n}_{\stackrel{i_1,\hdots,i_s=1}{|i_1 - i_2|\geq h+1,\hdots,|i_s - i_1|\geq h+1}} \Big| \frac{i_1 - i_2}{n}\Big|^{\alpha-2}\Big| \frac{i_2 - i_3}{n}\Big|^{\alpha-2}\hdots \Big| \frac{i_s - i_1}{n}\Big|^{\alpha-2} \frac{1}{n^s}
$$
\begin{equation}\label{e:mult_sum_goes_to_mult_integ_3/4<H<1}
\rightarrow \int^{1}_{0}\int^{1}_{0} \hdots \int^{1}_{0} |x_1 - x_2|^{\alpha-2}|x_2 - x_3|^{\alpha-2}\hdots |x_{s} - x_{1}|^{\alpha-2}dx_1 dx_2 \hdots dx_{s}, \quad n \rightarrow \infty.
\end{equation}
Indeed, since
$$
\sum^{n}_{\stackrel{i_1,\hdots,i_s=1}{i_1\neq i_2,i_2 \neq i_3,\hdots,i_s\neq i_1}} \Big| \frac{i_1 - i_2}{n}\Big|^{\alpha-2}\Big| \frac{i_2 - i_3}{n}\Big|^{\alpha-2}\hdots \Big| \frac{i_s - i_1}{n}\Big|^{\alpha-2} \frac{1}{n^s}
$$
\begin{equation}\label{e:sum_ik1_neq_ik2}
\rightarrow \int^{1}_{0}\int^{1}_{0} \hdots \int^{1}_{0} |x_1 - x_2|^{\alpha-2}|x_2 - x_3|^{\alpha-2}\hdots |x_{s} - x_{1}|^{\alpha-2}dx_1 dx_2 \hdots dx_{s}, \quad n \rightarrow \infty,
\end{equation}
and the sum on the left-hand side of \eqref{e:sum_ik1_neq_ik2} can be broken up into
$$
\Big\{ \sum^{n}_{\stackrel{i_1,\hdots,i_s=1}{|i_1 - i_2|\geq h+1,\hdots,|i_s - i_1|\geq h+1}} + \sum^{n}_{\stackrel{i_1,\hdots,i_s=1}{\{|i_1 - i_2|\geq h+1,\hdots,|i_s - i_1|\geq h+1}\}^c}\Big\} \Big| \frac{i_1 - i_2}{n}\Big|^{\alpha-2}\Big| \frac{i_2 - i_3}{n}\Big|^{\alpha-2}
$$
\begin{equation}\label{e:break_up_full_sum}
\hdots \Big| \frac{i_s - i_1}{n}\Big|^{\alpha-2} \frac{1}{n^s},
\end{equation}
then it suffices to show that the second summation term in \eqref{e:break_up_full_sum} goes to zero. However, the latter can be established by a similar argument to that in the proof of Lemma \ref{lem:3/4<H<1_neardiag_h}. Thus, \eqref{e:mult_sum_goes_to_mult_integ_3/4<H<1} holds.

Based on \eqref{e:gamma_h(z)}, recast the left-hand side of \eqref{e:mult_sum_gamma_goes_to_mult_integ_3/4<H<1} as
\begin{equation}\label{e:main_expression_large_indices}
\frac{h^{2s} }{\eta^s(n)\zeta^s(h)} \hspace{1mm} \sum^{n}_{\stackrel{i_1,i_2,\hdots,i_s =1}{|i_1 - i_2|\geq h+1,\hdots,|i_s - i_1|\geq h+1}}
|i_1 - i_2|^{\alpha-2}\{ \tau + g(i_1 - i_2,h)\}\hdots |i_s - i_1|^{\alpha-2}\{ \tau + g(i_s - i_1,h)\}.
\end{equation}
In view of \eqref{e:mult_sum_goes_to_mult_integ_3/4<H<1}, we only need to show that the remaining terms involving at least one residual function $g$ in \eqref{e:main_expression_large_indices} go to zero. Pick a number $\rho$ in the interval $(0, \min(\delta,\alpha-3/2))$. By \eqref{e:gamma_h(z)}, $\abs{g(h/z)}\leq C(h/z)^{\delta}\leq C(h/z)^{\rho}$, $z \geq h + 1$. Therefore,
$$
\frac{1}{n^{s(\alpha-1)}}\sum^{n}_{\stackrel{i_1,i_2,\hdots,i_s =1}{|i_1 - i_2|\geq h+1,\hdots,|i_s - i_1|\geq h+1}} |i_1 - i_2|^{\alpha-2} \hdots |i_s - i_1|^{\alpha-2} g( i_s - i_1,h)
$$
$$
\leq \frac{C}{n^{s(\alpha-1)}}\sum^{n}_{\stackrel{i_1,i_2,\hdots,i_s =1}{|i_1 - i_2|\geq h+1,\hdots,|i_s - i_1|\geq h+1}} |i_1 - i_2|^{\alpha-2} \hdots |i_s - i_1|^{\alpha-2} \abs{\frac{h}{i_s - i_1}}^{\rho}
$$
$$
= C \Big(\frac{h}{n}\Big)^{\rho}\sum^{n}_{\stackrel{i_1,i_2,\hdots,i_s =1}{|i_1 - i_2|\geq h+1,\hdots,|i_s - i_1|\geq h+1}} {\abs{\frac{i_1 - i_2}{n}}}^{\alpha-2} \hdots {\abs{\frac{i_{s-1} - i_s}{n}}}^{\alpha-2} {\abs{\frac{i_s - i_1}{n}}}^{\alpha-2-\rho} \frac{1}{n^s}
$$
\begin{equation}\label{e:3/4<H<1_one_g}
\sim C \Big(\frac{h}{n}\Big)^{\rho} \int^{1}_0  \hdots \int^{1}_0\abs{x_1 - x_2}^{\alpha-2} \hdots \abs{x_{s-1} - x_s}^{\alpha-2} \abs{x_s - x_1}^{\alpha-2-\rho} dx_1 \hdots dx_s \rightarrow 0,
\end{equation}
as $n\rightarrow\infty$. The limit in \eqref{e:3/4<H<1_one_g} is a consequence of \eqref{e:h(n)} and of the fact that the multiple integral is finite by the same argument as in Remark \ref{r:cs_is_finite}. This establishes \eqref{e:mult_sum_gamma_goes_to_mult_integ_3/4<H<1} under \eqref{e:gamma_h(z)}.

For $m > 1$, by \eqref{e:gamma=z2H-2_hk_hl_const+resid} and \eqref{e:eta_zeta} the constants $w_{k}$, $k=1,\hdots,m$, in \eqref{e:main_expression_large_indices} cancel out. Moreover, by \eqref{e:hk_neq_hl} and \eqref{e:hk_neq_hl_assumption_gamma_h_z<h}, the zero limit in \eqref{e:3/4<H<1_one_g} still holds; consequently, so does the limit \eqref{e:mult_sum_gamma_goes_to_mult_integ_3/4<H<1}. $\Box$\\

\noindent{\sc Proof of Lemma \ref{lem:0<H<3/4_neardiag_h}}
For $m = 1$, rewrite the sum in \eqref{e:0<H<3/4_neardiag_h_original_sum} as
\begin{equation}\label{e:0<H<3/4_neardiag_h}
\eta^{-s}(n)\zeta^{-s}(h)\Big\{ \sum_{\underset{\abs{i_1-i_2}\leq h\cup...\cup\abs{i_s-i_1}\leq h}{i_1,...,i_s=1}}^{n} + \sum_{\underset{\abs{i_1-i_2}\geq h+1,...,\abs{i_s-i_1}\geq h+1 }{i_1,...,i_s=1}}^{n}\Big\}
    \gamma_h({i_1-i_2})\cdots\gamma_h({i_s-i_1}).
\end{equation}
We will show that both multiple summation terms go to zero. We first show this over the index range $\abs{i_1-i_2}\leq h\cup...\cup\abs{i_s-i_1}\leq h$; moreover, as in the proof of Lemma \ref{lem:3/4<H<1_neardiag_h}, we will only consider the index set \eqref{e:summ_range_|i1-i2|=<h}.

Fix the parameter range $0 < \alpha < 3/2$. By \eqref{e:gamma_h(z)}, \eqref{e:hk_neq_hl_assumption_gamma_h_z<h} and the Cauchy-Schwarz inequality, the expression \eqref{e:3/4<H<1_exp_1} is bounded in absolute value by
$$
C \frac{h^{2(s-1)}}{n^{s/2}h^{s(\alpha+1/2)}}  \sum^{n}_{\stackrel{i_1,\hdots,i_s=1, i_1\neq i_2}{|i_1 - i_2|\leq h,|i_2 - i_3|\geq h+1, \hdots, |i_s - i_1|\geq h+1 }} h^{\alpha} \abs{i_2 - i_3}^{\alpha-2}\hdots \abs{i_s - i_1}^{\alpha-2}
$$
$$
\leq C \frac{h^{(2-\alpha)(s-1) - s/2}}{n^{s/2}}   \sum^{n}_{\stackrel{i_1,\hdots,i_{s-1}=1, i_1\neq i_2}{|i_1 - i_2|\leq h,|i_2 - i_3|\geq h+1, \hdots, |i_{s-2} - i_{s-1}|\geq h+1 }} \abs{i_2 - i_3}^{\alpha-2}\hdots \abs{i_{s-2} - i_{s-1}}^{\alpha-2}
$$
$$
\Big(\sum^{n}_{\stackrel{i_{s}=1}{|i_{s-1} - i_s|\geq h+1}}\abs{i_{s-1} - i_s}^{2(\alpha-2)}\Big)^{1/2}  \Big(\sum^{n}_{\stackrel{i_{s}=1}{|i_{s} - i_{1}|\geq h+1}}\abs{i_s - i_1}^{2(\alpha-2)}\Big)^{1/2}
$$
$$
\leq C \frac{h^{(2-\alpha)(s-1) - s/2}}{n^{s/2}}  \Big(\sum^{n}_{z=h}z^{2\alpha - 4}\Big)\sum^{n}_{\stackrel{i_1,\hdots,i_{s-1}=1, i_1\neq i_2}{|i_1 - i_2|\leq h,|i_2 - i_3|\geq h+1, \hdots, |i_{s-2} - i_{s-1}|\geq h+1 }} \abs{i_2 - i_3}^{\alpha-2}\hdots \abs{i_{s-2} - i_{s-1}}^{\alpha-2}
$$
\begin{equation}\label{e:overall_sum_under_one_restriction}
\leq C \frac{h^{(2-\alpha)(s-1) - s/2}}{n^{s/2}} h^{2\alpha - 3}\Big(\sum^{n}_{z=h}z^{\alpha - 2}\Big)^{s-3} (nh)
\leq C \frac{h^{(s-2)(3/2 - \alpha) + (\alpha-1)}}{n^{s/2-1}}\Big(\sum^{n}_{z=h}z^{\alpha - 2}\Big)^{s-3}.
\end{equation}
In the subranges $0 < \alpha < 1$, $\alpha = 1$, $1 < \alpha < 3/2$, \eqref{e:overall_sum_under_one_restriction} is bounded, respectively, by the expressions
$C(\frac{h}{n})^{s/2-1}$,
$$
C\Big(\frac{h}{n}\Big)^{s/2-1} \log^{s-3}(n) = \Big(\frac{h \log^2( n)}{n}\Big)^{s/2-1} \frac{1}{\log(n)},
$$
and $C(\frac{h}{n})^{(s-2)(3/2 - \alpha) + (\alpha-1)}$, all of which converge to zero as $n \rightarrow \infty$ under \eqref{e:h(n)} for $s \geq 3$.

Next consider the case $\alpha = 3/2$. By a simple adaptation of the procedure leading to \eqref{e:overall_sum_under_one_restriction}, we arrive at the bound
$$
C \frac{h^{1/2}}{n^{s/2-1}\log^{s/2}(n)}\Big(\sum^{n}_{z=h}z^{- 1}\Big)\Big(\sum^{n}_{z=h}z^{- 1/2}\Big)^{s-3} \leq
C \Big(\frac{h}{n}\Big)^{1/2}\frac{1}{\log^{s/2-1}(n)}\rightarrow 0, \quad n \rightarrow \infty.
$$
Therefore, in the parameter range $0 < \alpha \leq 3/2$, by extending the conclusion to the whole summation range of interest,
\begin{equation}\label{e:0<H<3/4_gamma_sum}
\eta^{-s}(n) \zeta^{-s}(h) \sum_{\underset{\abs{i_1-i_2}\leq h\cup...\cup\abs{i_s-i_1}\leq h}{i_1,...,i_s=1}}^{n}
    \gamma_h({i_1-i_2})\cdots\gamma_h({i_s-i_1})\rightarrow 0, \quad n\rightarrow \infty.
\end{equation}

We now show that the multiple summation over the index range $\abs{i_1-i_2}\geq h,...,\abs{i_s-i_1}\geq h$ in \eqref{e:0<H<3/4_neardiag_h} also goes to zero. Starting from the expression \eqref{e:main_expression_large_indices}, by the same argument with the residual function $g$ in the proof of Lemma \ref{lem:3/4<H<1_offdiag_h}, it suffices to consider
$$
\Big( \frac{h^2}{\eta(n) \zeta(h) }\Big)^s \sum^{n}_{\stackrel{i_1,i_2,\hdots,i_s =1}{|i_1 - i_2|\geq h+1, |i_2 - i_3|\geq h+1 , \hdots , |i_s - i_1|\geq h+1}}
|i_1 - i_2|^{\alpha-2}|i_2 - i_3|^{\alpha-2} \hdots |i_s - i_1|^{\alpha-2}.
$$
By Cauchy-Schwarz, this expression is bounded from above by
$$
\frac{h^{2s}}{\eta^s(n)\zeta^s(h)} \sum^{n}_{\stackrel{i_1,i_2,\hdots,i_{s-1}=1}{|i_1 - i_2|\geq h+1,\hdots,|i_{s-2} - i_{s-1}|\geq h+1}} |i_1 - i_2|^{\alpha-2}|i_2 - i_3|^{\alpha-2}\hdots |i_{s-2} - i_{s-1}|^{\alpha-2}
$$
$$
\Big(\sum^{n}_{\stackrel{i_s=1}{|i_{s-1}-i_s|\geq h+1}}|i_{s-1}-i_s|^{2\alpha-4}\Big)^{1/2}
\Big(\sum^{n}_{\stackrel{i_s=1}{|i_{s}-i_1|\geq h+1}}|i_{s}-i_1|^{2\alpha-4}\Big)^{1/2}
$$
\begin{equation}\label{e:bound_overall_sum_only_restrictions_>h+1_after_CS}
\leq \frac{C h^{2s}}{\eta^s(n)\zeta^s(h)} \bigg( \sum_{z=h+1}^{n}z^{2\alpha-4}  \bigg) \sum^{n}_{\stackrel{i_1,i_2,\hdots,i_{s-1}=1}{|i_1 - i_2|\geq h+1,\hdots,|i_{s-2} - i_{s-1}|\geq h+1}}
|i_1 - i_2|^{\alpha-2}|i_2 - i_3|^{\alpha-2}\hdots |i_{s-2} - i_{s-1}|^{\alpha-2}.
\end{equation}
However, the multiple summation term in \eqref{e:bound_overall_sum_only_restrictions_>h+1_after_CS} is bounded by
$$
C \Big\{\sum^{n}_{z=h+1}z^{\alpha-2}\Big\}^{s-3} \sum^{n}_{\stackrel{i_1,i_2=1}{|i_1-i_2|\geq h+1}} |i_1 - i_2|^{\alpha-2}
$$
\begin{equation}\label{e:sum_z_sum_i_bound}
\leq C' \Big\{\sum^{n}_{z=h+1}z^{\alpha-2}\Big\}^{s-3} n \sum^{n}_{z=h+1}\Big( 1 - \frac{z}{n}\Big) z^{\alpha-2} \leq C n \Big\{\sum^{n}_{z=h+1}z^{\alpha-2}\Big\}^{s-2}.
\end{equation}
Therefore, when $\alpha = 3/2$, by \eqref{e:sum_z_sum_i_bound} the expression \eqref{e:bound_overall_sum_only_restrictions_>h+1_after_CS} can be bounded by
$$
C \frac{\log(n) n n^{(s-2)/2}}{n^{s/2}\log^{s/2}(n)}
= C \frac{1}{\log^{s/2-1}(n)} \rightarrow 0, \quad n\rightarrow \infty,
$$
as $n \rightarrow \infty$, since $s \geq 3$ and by \eqref{e:h(n)}.

On the other hand, when $0 < \alpha < 1$, $\alpha = 1$ and $1 < \alpha < 3/2$, the bound for \eqref{e:bound_overall_sum_only_restrictions_>h+1_after_CS} becomes, respectively,
\begin{equation}\label{e:0<H<3/4_limit1}
C \frac{h^{2s}}{h^{s(\alpha+1/2)}} \frac{n}{n^{s/2}} h^{(\alpha-1)(s-2)} h^{2\alpha-3} =C\Big(\frac{h}{n}\Big)^{s/2-1} \rightarrow 0,
\end{equation}
\begin{equation}\label{e:0<H<3/4_limit2}
C \frac{h^{2s}}{h^{3s/2}} \frac{n}{n^{s/2}}  \log^{s-2}(n) h^{-1} =C\Big(\frac{h\log^2(n)}{n}\Big)^{s/2-1} \rightarrow 0,
\end{equation}
and
\begin{equation}\label{e:0<H<3/4_limit3}
C \frac{h^{2s}}{h^{s(\alpha+1/2)}} \frac{n}{n^{s/2}} n^{(\alpha-1)(s-2)} h^{2\alpha-3} =C\Big(\frac{h}{n}\Big)^{(s-2)(3/2-\alpha)} \rightarrow 0,
\end{equation}
as $n \rightarrow \infty$. These three limits hold because $s \geq 3$ and by \eqref{e:h(n)}. Thus, the expressions \eqref{e:0<H<3/4_gamma_sum}, \eqref{e:0<H<3/4_limit1}, \eqref{e:0<H<3/4_limit2}, and \eqref{e:0<H<3/4_limit3} yield \eqref{e:0<H<3/4_neardiag_h_original_sum} for $m = 1$.

For $m > 1$, by \eqref{e:hk_neq_hl} and \eqref{e:hk_neq_hl_assumption_gamma_h_z<h} the zero limits in \eqref{e:0<H<3/4_gamma_sum}, \eqref{e:0<H<3/4_limit1}, \eqref{e:0<H<3/4_limit2}, and \eqref{e:0<H<3/4_limit3} still hold; consequently, so does \eqref{e:0<H<3/4_neardiag_h_original_sum}. $\Box$\\

\noindent{\sc Proof of Lemma \ref{lem:0<H<3/4_s=2}} We begin by showing ($i$) for $m = 1$. Rewrite
\begin{equation}\label{e:0<H<3/4_sum_gamma2}
\eta^{-2}(n)\zeta^{-2}(h)\sum_{i_1,i_2=1}^{n} \gamma_{h}^2(i_1-i_2)
= \sum_{z=-n+1}^{n-1} \bigg(1-\frac{\abs{z}}{n}\bigg) (h^{-\alpha}\gamma_{h}(z))^2 \frac{1}{h}.
\end{equation}
As $n\rightarrow \infty$, the summand in \eqref{e:0<H<3/4_sum_gamma2} goes to, and is also bounded by, $(h^{-\alpha}\gamma_{h}(z))^2 \frac{1}{h}$. Therefore, if we can show that
\begin{equation}\label{e:0<alpha<3/2_sum_target_expression}
\sum_{z=-n+1}^{n-1}  (h^{-\alpha}\gamma_{h}(z))^2 \frac{1}{h} \rightarrow \Big(\frac{C_{\alpha}}{C_H}\Big)^4  \|\widehat{G}(x)\|^2_{L^2(\bbR)}, \quad n \rightarrow \infty,
\end{equation}
then \eqref{e:0<H<3/4_s=2} is obtained as a consequence of the dominated convergence theorem. Indeed, by setting $w_k=w_l=1$ and making the change of variables $hx = y$ in the relation \eqref{e:gamma_kl/h^alpha},
\begin{equation}\label{e:deviation_gamma(z)/h^alpha-integral}
    \abs{\frac{\gamma_{h}(z)}{h^{\alpha}} - \Big(\frac{C_{\alpha}}{C_{H}}\Big)^2 C^2_H \int_{\bbR} e^{iy z/h}
    \frac{\abs{e^{iy}-1}^2}{\abs{y}^{\alpha+1}} dy }\leq C h^{-\delta},\quad h,z\in\bbZ,\, \quad h\geq \varepsilon_0^{-2}.
\end{equation}
Therefore,
\begin{equation}\label{e:h^(-alpha)gammah(z)}
h^{-\alpha}\gamma_{h}(z) = \Big(\frac{C_{\alpha}}{C_{H}}\Big)^2 {G}_{H}\Big(\frac{z}{h}\Big)+O(h^{-\delta}),
\end{equation}
where $G_{H}$ denotes the covariance function of a standard fractional Gaussian noise (fGn) $Y(t) = B_{H}(t) - B_{H}(t-1)$, $t \in \bbR$, i.e.,
$$
G_{H}(z):= EY(t)Y(t+z) = \frac{\abs{1+z}^{2H}-2\abs{z}^{2H}+ \abs{1-z}^{2H} }{2}, \quad z \in \bbR.
$$
So, recast the expression on the left-hand side of \eqref{e:0<alpha<3/2_sum_target_expression} as
\begin{equation}\label{e:0<alpha<3/2_sum_squares_cov}
\Big(\frac{C_{\alpha}}{C_H}\Big)^4  \sum_{z=-n+1}^{n-1} G^2_{H}\bigg(\frac{z}{h}\bigg) \frac{1}{h}
+ C \hspace{0.5mm}O\bigg( \frac{1}{h^{1+\delta}}\bigg)\sum_{z=-n+1}^{n-1}  G_{H}\bigg(\frac{z}{h}\bigg)  + o(1),
\end{equation}
where the vanishing term $o(1)$ is a consequence of \eqref{e:h(n)}. Since $G_H(z)\in L^{2}(\bbR)$ for $0 < H< 3/4$ ($0 < \alpha < 3/2$; see \eqref{e:alpha}), the first summation on the right-hand side of \eqref{e:0<alpha<3/2_sum_squares_cov} converges to
\begin{equation}\label{e:0<alpha<3/2_sum_cov_limiting_constant}
 \Big(\frac{C_{\alpha}}{C_H}\Big)^4 \int_{\bbR} G^2_{H}(z) dz = \Big(\frac{C_{\alpha}}{C_H}\Big)^4 \int_{\bbR} |\widehat{G}(x)|^2 dx.
\end{equation}
The equality in \eqref{e:0<alpha<3/2_sum_cov_limiting_constant} is a consequence of Parseval's theorem based on the inverse Fourier transform $f(z) = (2 \pi)^{-1/2} \int_{\bbR} e^{izx}\widehat{f}(x) dx$, $f \in L^2(\bbR)$. Moreover,
\begin{equation}\label{e:(1/h^(1+delta))sum_G(z/h)}
O\Big(\frac{1}{h^{1 + \delta}}\Big) \sum^{n-1}_{z = -n+1}G_H\Big( \frac{z}{h}\Big) \rightarrow 0, \quad n \rightarrow \infty,
\end{equation}
since the function $G_H(\cdot)$ is bounded and by \eqref{e:h(n)}. So, by the expressions \eqref{e:0<alpha<3/2_sum_squares_cov}, \eqref{e:0<alpha<3/2_sum_cov_limiting_constant} and \eqref{e:(1/h^(1+delta))sum_G(z/h)}, we obtain \eqref{e:0<alpha<3/2_sum_target_expression}, and hence \eqref{e:0<H<3/4_s=2}, for $m = 1$.

For $m > 1$, essentially the same argument can be used, and we simply indicate the minor changes. The expression \eqref{e:0<H<3/4_sum_gamma2} must be replaced by
$$
\eta^{-2}(n)\zeta^{-1}(h_{k_1})\zeta^{-1}(h_{k_2})\sum_{i_1,i_2=1}^{n} \gamma_{h}^2(i_1-i_2)_{k_1,k_2}
= \frac{1}{(w_{k_1}w_{k_2})^{\alpha+1/2}}\sum_{z=-n+1}^{n-1} \bigg(1-\frac{\abs{z}}{n}\bigg) (h^{-\alpha}\gamma_{h}(z)_{k_1,k_2})^2 \frac{1}{h}.
$$
In addition, in expression \eqref{e:deviation_gamma(z)/h^alpha-integral} one should substitute $C^{2}_{H}\int_{\bbR} e^{iyz/h} (e^{iw_{k_1}y}-1)(e^{iw_{k_2}y}-1) |y|^{-(\alpha + 1)}dy$ for the integral $C^{2}_{H}\int_{\bbR} e^{iyz/h} |e^{iy}-1|^2 |y|^{-(\alpha + 1)}dy$, where the former can be reinterpreted as the covariance between the increments $B_{H}(t) - B_{H}(t-w_{k_1})$ and $B_{H}(t') - B_{H}(t'-w_{k_2})$, $t -t' = \frac{z}{h}$, of a standard fBm $B_H$. The rest of the argument can be applied in the same way to eventually arrive at the limit \eqref{e:0<alpha<3/2_sum_target_expression} with $\norm{\widehat{G}(y;w_{k_1},w_{k_2})}^2_{L^2(\bbR)}$ in place of $\|\widehat{G}(x)\|^2_{L^2(\bbR)}$. Thus, \eqref{e:0<H<3/4_s=2} holds also for $m > 1$. \\

To show ($ii$) for $m =1$, note that we can apply \eqref{e:hk_neq_hl_assumption_gamma_h_z<h} with $\alpha = 3/2$ in the summation range $|i_1 - i_2|\leq h$ to obtain
\begin{equation}\label{e:double_sum_cov2_alpha=3/2_restricted_range}
n^{-1}\log^{-1}(n) h^{-4}\sum^{n}_{\stackrel{i_1,i_2 =1}{|i_1 - i_2|\leq h}} \gamma^2_h(i_1 - i_2) \leq \frac{C}{\log(n)} \rightarrow 0,
\end{equation}
by \eqref{e:h(n)}. Alternatively, in the summation range $\abs{i_1-i_2}\geq h+1$, by \eqref{e:gamma=z2H-2_hk_hl_const+resid} we have
$$
\zeta^{-2}(h)\eta^{-2}(n) \sum_{\underset{\abs{i_1-i_2}\geq h+1}{i_1,i_2=1}}^{n} \gamma_h^2({i_1-i_2})
= 2 n^{-1} \log^{-1}(n) h^{-4} \sum_{z=h+1}^{n-1} n\Big(1-\frac{z}{n}  \Big)z^{-1} h^4 \{\tau+g(z,h)\}^2
$$
$$
\sim 2 \log^{-1}(n) \sum_{z=h+1}^{n-1} z^{-1} \{\tau+g(z,h)\}^2
$$
\begin{equation}\label{eqn:0<h_leq_3/4_z>h_0_2}
=  2\log^{-1}(n) \Big\{\sum_{z=h+1}^{n-1} z^{-1} \tau^2 + \sum_{z=h+1}^{n-1} z^{-1}g^2(z,h)
+ 2 \sum_{z=h+1}^{n-1} z^{-1} \tau g(z,h) \Big\}.
\end{equation}
Note that for $\beta > 0$ and large enough $n$, $\int^{n}_{h+1} z^{-\beta}dz \leq  \sum^{n-1}_{z= h+1}z^{-\beta}  \leq \int^{n}_{h+1} (z-1)^{-\beta}dz$. Consequently, if $\beta=1$,
$$
\frac{\log (n)-\log(h+1)}{\log(n)} \leq \frac{\sum^{n-1}_{z= h+1}z^{-1}}{\log(n)} \leq \frac{\log (n-1)-\log(h)}{\log(n)}.
$$
Thus, the left summation term in \eqref{eqn:0<h_leq_3/4_z>h_0_2} goes to $2\tau^2$ as $n\rightarrow \infty$. We now show that the remaining two terms in \eqref{eqn:0<h_leq_3/4_z>h_0_2} go to zero with $n$. It also suffices to look at the third term in \eqref{eqn:0<h_leq_3/4_z>h_0_2}, because a similar approach can be used with the second term. Indeed, the former can be bounded by
$$
\abs{ \frac{C}{\log(n)} \sum_{z=h+1}^{n-1} z^{-1} \tau g(z,h) }
\leq \frac{C'}{\log(n)}  \sum^{n-1}_{z=h+1} \Big(\frac{z}{h}\Big)^{-(1+\delta)} \frac{1}{h}
\leq \frac{C'}{\log(n)} \int_{1}^{\infty} x^{-(1+\delta)}dx \rightarrow 0,
$$
as $n \rightarrow \infty$. Together with \eqref{e:double_sum_cov2_alpha=3/2_restricted_range}, this establishes \eqref{e:H=3/4_s=2} for $m =1$.

For $m > 1$, by \eqref{e:gamma=z2H-2_hk_hl_const+resid} and \eqref{e:eta_zeta} the constants $w_{k}$, $k=1,\hdots,m$, in \eqref{eqn:0<h_leq_3/4_z>h_0_2} cancel out. Moreover, by \eqref{e:hk_neq_hl} and \eqref{e:hk_neq_hl_assumption_gamma_h_z<h}, the zero limits in \eqref{e:double_sum_cov2_alpha=3/2_restricted_range} and in \eqref{eqn:0<h_leq_3/4_z>h_0_2} (for the second and third terms) still hold; consequently, so does the limit \eqref{e:H=3/4_s=2}. $\Box$\\

\bibliography{MSD_asympt_Gaussian_si}

\small

\bigskip

\noindent \begin{tabular}{l}
Gustavo Didier and Kui Zhang\\
Mathematics Department\\
Tulane University  \\
6823 St.\ Charles Avenue\\
New Orleans, LA 70118, USA\\
\textit{gdidier@tulane.edu}, \hspace{1mm}\textit{kzhang3@tulane.edu} \\
\end{tabular}\\

\smallskip

\end{document}